\documentclass{article}
\pdfoutput=1
\usepackage[pdfencoding=auto, psdextra, pagebackref, breaklinks]{hyperref}

\usepackage{L_func}
\usepackage[title]{appendix}

\title{Exotic Monoidal Structures and Abstractly Automorphic Representations for $\mathrm{GL}(2)$}
\author{Gal Dor\\
Tel-Aviv University}
\date{September 2022}

\begin{document}

\maketitle

\begin{abstract}
    We use the theta correspondence to study the equivalence between Godement--Jacquet and Jacquet--Langlands $L$-functions for $\GL(2)$. We show that the resulting comparison is in fact an expression of an exotic symmetric monoidal structure on the category of $\GL(2)$-modules.
    
    Moreover, this enables us to construct an abelian category of \emph{abstractly automorphic} representations, whose irreducible objects are the usual automorphic representations. We speculate that this category is a natural setting for the study of automorphic phenomena for $\GL(2)$, and demonstrate its basic properties.
    
    This paper is a part of the author's thesis \cite{alg_struct_L_funcs}. 
\end{abstract}

\tableofcontents

\section{Introduction} \label{sect:intro_to_triality}

In order to minimize technical difficulties, in this paper we will only discuss automorphic forms over function fields $F$ of characteristic $\neq 2$. Denote the adeles over $F$ by $\AA=\AA_F$.

This paper has two parts. The first part focuses on the construction of a symmetric monoidal structure on the category of smooth representations of $\GL(2)$ over a local field. The second part uses this construction in the global setting, to construct an abelian category of ``abstractly automorphic'' representations.

While the first part takes up the majority of this paper, it is the author's expectation that the second part (which uses the first in an essential way) would be the most interesting.

Therefore, before delving into the details of the first part of the paper, let us briefly discuss the second.

There is a famously unsatisfactory aspect of the global theory of automorphic representations. There is an informal analogy between the global cuspidal representations and local supercuspidal representations. Both are in some sense compactly supported: the local representations in terms of matrix coefficients, and the global representations in terms of their corresponding automorphic forms. Moreover, there is some similarity between the local theory of Jacquet modules and the global theory of constant terms. Additionally, there is similarity between the notion of parabolic induction and the notion of Eisenstein series.

This analogy between the local and global theory is fairly imprecise. Supercuspidal representations, Jacquet modules and parabolic induction all enjoy the status of functors, with clear formal properties and relationships between each other. On the other hand, cusp forms, constant terms and Eisenstein series are all considered at a lower level of categorification, as maps between specific spaces of automorphic functions.

We want to strengthen this analogy by placing these constructions on an equal footing.

That is, we want to construct a category of ``abstractly automorphic'' representations. This category should decompose as a category into a cuspidal and an Eisenstein part. Constant terms and Eisenstein series should translate into a pair of adjoint functors, and cuspidal abstractly automorphic representations should all be killed by taking their automorphic parabolic restriction. This should allow us to discuss the cuspidal automorphic spectrum in more or less the same terms as the supercuspidal local spectrum.

This paper will culminate in the construction (and proving the basic properties) of a category of abstractly automorphic representations for $\GL(2)$ over $F$. A detailed proof of the above properties appears in a separate paper \cite{parabolic_ind_arxiv}.

The remainder of this introduction is structured as follows. In Subsection~\ref{subsect:intro_symm_mon}, we will describe in detail the ideas comprising Part~\ref{part:monoidal_structure}, as well as provide its outline. In Subsection~\ref{subsect:intro_abst_aut_cat}, we will describe the ideas and outline of Part~\ref{part:abst_aut_cat}. Finally, in Subsection~\ref{subsect:intro_add_consequences}, we will briefly describe Appendix~\ref{app:spin_struct}, which contains an application of the results of Part~\ref{part:abst_aut_cat}.

\date{\textbf{Acknowledgements: }The author would like to thank his advisor, Joseph Bernstein, for his insights and ideas about important problems. The author would also like to thank Shachar Carmeli for helping refine the ideas in this paper, as well as both Shachar Carmeli and Yiannis Sakellaridis for their great help improving the quality of this text.}

\subsection{Symmetric Monoidal Structure} \label{subsect:intro_symm_mon}

Let us begin with describing in detail the ideas in the first part of this paper.

Let $(\pi,V)$ be an automorphic representation for a reductive group $G$ over $F$. Automorphic $L$-functions for $V$ are usually defined via one of many different procedures. These procedures all share the following general form:
\begin{enumerate}
    \item Take a space of test functions (e.g., $S(\M_2(\AA))$, the space of matrix coefficients of $V$, the space of vectors from $V$, etc.).
    \item To each test function, assign a \emph{zeta integral} -- a meromorphic function in a complex variable $s$ in some right-half-plane.
    \item Take the GCD of the resulting zeta integrals, guided by some place-by-place prescription to deal with the overall normalization.
\end{enumerate}
The GCD we end up with is a meromorphic function in $s$, known as the \emph{$L$-function} for $(\pi,V)$.

The procedure above might seem rather ad hoc. Nevertheless, it gives correct results that conform to our expectations. The resulting $L$-functions admit analytic continuation, satisfy a functional equation, and match up with those produced from Galois representations (where the procedure to produce them is much more straight-forward).

The research detailed in this paper was the result of taking a different point of view on these automorphic $L$-functions. The idea is to study the entire space of zeta integrals, instead of just the $L$-functions that generate them. From this point of view, it is the collection of zeta integrals itself that is important, along with the relationships between such collections. This is analogous to the way in which a vector space can be studied without necessarily specifying a concrete basis: the space of zeta integrals is analogous to the vector space, and the $L$-function is a basis, or a generator, for this space.

This paper is the first in a series of papers, all based on this same philosophical idea, giving a variety of results.

In this paper specifically, we will demonstrate some hidden algebraic structure that can come out of comparing different spaces of zeta integrals. We will look at automorphic representations for $\GL(2)$, and specifically, the comparison of Godement--Jacquet and Jacquet--Langlands zeta integrals. These are two completely different constructions of the $L$-function for $\GL(2)$, both of which give the same result. We will will extend this to a canonical bijection between their spaces of zeta integrals, instead of just their $L$-functions. It will turn out that this point of view, where the focus is on the zeta integrals, has unexpected benefits.

To be more precise, this bijection between zeta integrals turns out to carry an algebraic structure: it induces a novel multiplicative structure on representations of $\GL(2)$. The main focus of this paper will be refining this algebraic structure and studying some of its consequences.

Let us recall each of the Godement--Jacquet and Jacquet--Langlands constructions in turn. Suppose that we are given an automorphic representation $(\pi,V)$. For Godement--Jacquet $L$-functions, one takes a smooth and compactly supported test function $\Psi\in S(\M_2(\AA))$, and integrates it along with a matrix coefficient $\beta\in V\otimes\widetilde{V}$ to obtain the zeta integral:
\[
    Z_\text{GJ}\left(\beta,\Psi,s\right)=\int_{\GL_2(\AA)} \beta(g)\Psi(g)\abs{\det(g)}^{s+\frac{1}{2}}\dtimes{g}.
\]
Here, $\widetilde{V}$ denotes the contragradient, and $\beta(g)$ denotes the contraction of $\pi(g)$ with the vector and co-vector defining $\beta$. The GCD of all these zeta integrals in an appropriate ring of holomorphic functions of $s\in\CC$ is the Godement--Jacquet $L$-function.

However, there is an alternative approach to defining $L$-functions for $\GL(2)$, by Jacquet--Langlands in the adelic formulation, based on a formulation by Hecke in more classical language (see also Remark~\ref{remark:hecke}). In this case, one takes an automorphic form $\phi\in V$ and integrates it along some subgroup to obtain the zeta integral:
\[
    Z_\text{JL}\left(\phi,s\right)=\int_{\AA^\times/F^\times} \phi\left(\begin{pmatrix}y & \\ & 1\end{pmatrix}\right)\abs{y}^{s-\frac{1}{2}}\dtimes{y}.
\]

\begin{remark} \label{remark:hecke}
    Let us make a couple of brief remarks about terminology. As stated above, the first variant of the zeta integrals we are referring to as ``Jacquet--Langlands'' zeta integrals was introduced by Hecke in the classical language of modular forms. Later, Jacquet and Langlands gave these integrals an adelic and local interpretation. Since in this text, we are emphasizing the local case and the role of the adelic language, the author has chosen to refer to these zeta integrals as ``Jacquet--Langlands'' zeta integrals, rather than Hecke zeta integrals (which might be further confused with the $L$-functions of Hecke \emph{characters}).
    
    In a similar fashion, for the sake of historical accuracy, one should also remark that what we are referring to as ``Godement--Jacquet'' zeta integrals are generalizations to $\GL(2)$ of the ideas laid out by Tate in his thesis. Tate's ideas were originally stated in the adelic language, but only for $\GL(1)$.
\end{remark}

Let us go back to our main line of discussion and elaborate on the structure relating the two kinds of zeta integrals. Our proof of the equivalence between Godement--Jacquet and Jacquet--Langlands $L$-functions will directly associate the data of a specific Godement--Jacquet zeta integral (a pair $\Psi\otimes\beta$) with the data of a specific Jacquet--Langlands zeta integral (a form $\phi$). This isomorphism of ``modules'' of zeta integrals is a categorification of the usual equivalence of $L$-functions.

Given the description above, one might be led to suspect the existence of some kind of \emph{algebra structure} on the space $\cS=S(\GL_2(F)\backslash\GL_2(\AA))$ of automorphic forms. After all, we are associating pairs of forms $\beta\in V\otimes\widetilde{V}$ with single forms $\phi\in V$. Amazingly, that turns out to more or less be the case! There is, however, one caveat: the symmetric monoidal structure that one needs to use on the category of representations of $\GL_2(\AA)$ is not the standard one, but instead a construction (related to the theta correspondence) using the space of test functions on $\M_2(\AA)$.

The author would like to note that the papers \cite{Y_times_of_triv} and in particular \cite{GJ_L_funcs_vs_test_functions} study aspects of the theta correspondence that are related to this symmetric monoidal structure (although without referring to it as such).

The construction of this symmetric monoidal structure is the topic of the first part of this paper, while the algebra of automorphic functions is constructed in the second part.

A rough outline of Part~\ref{part:monoidal_structure} is as follows. We will begin by establishing the aforementioned correspondence between Godement--Jacquet and Jacquet--Langlands $L$-functions in Section~\ref{sect:GJ_vs_JL}, focusing on the local picture. This will be done by looking at the space $Y=S(\M_2(F)\times F^\times)$ of smooth compactly supported functions, where $F$ is a non-Archimedean local field. This space carries a left $G=\GL_2(F)$-action by left-multiplication on $\M_2(F)$ and the determinant action of $G$ on $F^\times$. Similarly, it also carries a right $G$-action by right-multiplication on $\M_2(F)$ and the determinant action on $F^\times$. The key observation is that $Y$ additionally carries a third, \emph{hidden}, $G$-action that commutes with these two, coming from the Weil representation. The third action will allow us to directly relate the two kinds of zeta integrals. 

In Section~\ref{sect:T_tensor}, we will try to understand the object $Y$. It turns out that a useful way to do this is to change our perspective, and think of the object $Y$ as defining a tensor product operation given by:
\[
    V\oY V'\coloneqq V\otimes_G Y \otimes_G V',
\]
where $\otimes_G$ denotes the relative tensor product over $G$, i.e. the universal space through which $G$-invariant pairings factor. The operation $\oY$ turns out to be a unital, associative, symmetric monoidal structure on the category $\Mod(G)$ of smooth representations of $\GL_2(F)$. In fact, the structure of $\Mod(G)$ equipped with $\oY$ is slightly richer: $\Mod(G)$ can be turned into a commutative Frobenius algebra in $\CC$-linear presentable categories, in an appropriately categorified sense.

\subsection{Abstractly Automorphic Representations} \label{subsect:intro_abst_aut_cat}

Let us describe in detail the ideas in the second part of this paper.

In Part~\ref{part:abst_aut_cat}, we will go back to the global picture. We will show that under the symmetric monoidal structure $\oY$, the space $\cS=S(\GL_2(F)\backslash\GL_2(\AA))$ of automorphic functions now carries a multiplication operation, via a kind of theta lifting. The co-domain of this multiplication is not quite $\cS$. However, if we denote by $\cI$ the subspace of $\cS$ which is orthogonal to all functions of the form $\chi(\det(g))$, then $\cI$ is closed under multiplication. This turns out to induce on $\cI$ a structure of unital associative commutative algebra with respect to $\oY$. This will give us an algebra of automorphic functions.

We will use this construction, the algebra structure on $\cI$, to construct a category of what we refer to as abstractly automorphic representations. Specifically, given a monoidal category $(\C,\otimes)$ and an algebra object $A\in\C$ with respect to its monoidal structure $\otimes$, one can construct the category of $A$-modules in $\C$. This is given by the category of objects $M\in\C$, along with action maps
\[
    a\co A\otimes M\ra M
\]
satisfying the usual axioms.

By applying this construction to the algebra object $\cI$ with respect to the monoidal structure $\oY$, we obtain the category of $\cI$-modules in $\Mod(\GL_2(\AA))$, which we denote by $\Mod^\aut(\GL_2(\AA))$, and refer to as the category of \emph{abstractly automorphic representations}.

The category $\Mod^\aut(\GL_2(\AA))$ is canonically equipped with a forgetful functor
\[
    \Mod^\aut(\GL_2(\AA))\ra\Mod(\GL_2(\AA)),
\]
which turns out to be fully faithful in our case. This means that being an abstractly automorphic representation (or, more concretely, an $\cI$-module) is a mere \emph{property} of a $\GL_2(\AA)$-module. We choose to interpret this property as an automorphicity property for $\GL_2(\AA)$-modules, which we call \emph{abstract automorphicity}.

The question of what, exactly, is an automorphic representation is an old one. It is generally accepted that an irreducible representation is automorphic if it is a subquotient of an appropriate space $\widetilde{\cS}$ of functions on $\GL_2(F)\backslash\GL_2(\AA)$ (this is, for example, the definition used in \cite{what_is_aut_rep}). However, these irreducible automorphic representations do not seem to be known to fit into any natural category.

We propose that the category $\Mod^\aut(\GL_2(\AA))$ of abstractly automorphic representations can be taken as an answer to this question. A key property of the abelian category $\Mod^\aut(\GL_2(\AA))$, which helps justify its name, is that its irreducible objects are precisely the irreducible automorphic representations in the above sense (although the proof of this result lies outside the scope of this paper, see \cite{parabolic_ind_arxiv} for the proof). In particular, this seems to allow the use of more categorical methods in the theory of automorphic representations.

The detailed study of the category $\Mod^\aut(\GL_2(\AA))$ lies outside the scope of this paper. Nevertheless, it appears to be much better behaved than the ambient category $\Mod(\GL_2(\AA))$ of smooth $\GL_2(\AA)$-modules. For example, although we will only show this in a follow-on paper \cite{parabolic_ind_arxiv}, the category $\Mod^\aut(\GL_2(\AA))$ admits a decomposition into components that are either induced from $\GL(1)\times\GL(1)$ (corresponding to Eisenstein series) or cuspidal in nature (corresponding to cuspidal automorphic representations), in analogy to the local theory of representations of $\GL(2)$.

The category $\Mod^\aut(\GL_2(\AA))$ will be described in Subsection~\ref{subsect:Y_alg_sheaves}.

\begin{example}
    One way to think of the multiplicative structure on $\cI$ is to compare it to the commutative case.
    
    Let $H$ be a commutative group, and let $\Gamma\subseteq H$ be a subgroup. Then then the space of functions on $\Gamma\backslash H$ admits an algebra structure, given by the \emph{convolution product} with respect to $H$.
    
    If we think of $H$ as being analogous to $\GL_2(\AA)$, and $\Gamma$ as being analogous to $\GL_2(F)$, then we can justify thinking of the algebra structure of $\cI$ as being a kind of convolution product. In a sense, the non-standard symmetric monoidal structure $\oY$ allows inherently commutative phenomena to manifest for the non-commutative group $\GL_2(\AA)$. See also Remarks~\ref{remark:oY_is_like_comm} and~\ref{remark:mult_is_like_conv}.
\end{example}

\begin{remark}
    Let us make a much more speculative remark.
    
    The author believes that it would be interesting to generalize the results of this paper to groups other than $\GL(2)$. At the very least, if $B$ is a quaternion algebra over $F$, then it is possible to place the theta correspondence between $B^\times$ and $\GL(2)$ into the same framework of symmetric monoidal and module structures as this paper. The author intends to do so in a subsequent paper.
    
    Regardless, the author has not yet been able to take this theory beyond $\GL(2)$ and its inner forms. However, it is possible to formulate some expectations for such a generalization. Essentially, for a reductive group $G$, one can hope to construct some category $\Mod^\aut(G(\AA))$ of abstractly automorphic $G(\AA)$-modules. We can expect that there be a realization functor:
    \[
        \Mod^\aut(G(\AA))\ra\Mod(G(\AA)),
    \]
    but it seems too harsh to demand that the realization functor be fully faithful outside the case of $G=\GL(n)$.
\end{remark}

\subsection{Additional Consequences} \label{subsect:intro_add_consequences}

Before ending this introduction, let us mention one more interesting consequence of our construction of the algebra $\cI$ of automorphic functions. Consider the space $\cI_{/\SL_2(\O_\AA)}$ of co-invariants of $\cI$ with respect to the action of $\SL_2(\O_\AA)$, where $\O_\AA\subseteq\AA$ is the subring of all adeles which are integral at all places. Note that $\cI_{/\SL_2(\O_\AA)}$ is the space of all smooth compactly supported \emph{spherical} automorphic functions that are orthogonal to all one-dimensional characters (note that because $\SL_2(\O_\AA)$ is compact, it does not matter whether we take invariants or co-invariants).

The algebra structure of the space $\cI$ with respect to $\oY$ will, as a special case, automatically equip the space $\cI_{/\SL_2(\O_\AA)}$ with the structure of a commutative algebra with respect to the \emph{usual} tensor product $\otimes$ (up to choice of a piece of global data called an unramified spin structure, see Appendix~\ref{app:spin_struct}). This algebra structure is different from the usual point-wise multiplication, and respects the action of the center of the category $\Mod(\GL_2(\AA))$. This can be thought of as a kind of convolution product for spherical automorphic functions, in the following sense.

The space $\cI_{/\SL_2(\O_\AA)}$ of spherical automorphic functions has a Fourier transform. Specifically, the Hecke operators act on this space, and therefore it defines a module on the spectrum of all Hecke operators. Classically, one can try to identify this module with the structure sheaf of its support on the spectrum of the Hecke operators, and thus associate to each function $\phi\in\cI_{/\SL_2(\O_\AA)}$ a function on the unramified automorphic spectrum. However, this association is non-canonical. This manifests as our freedom to choose different unramified vectors at each unramified automorphic representation. In particular, it makes no sense to talk about the convolution of two functions $\phi,\phi'\in\cI_{/\SL_2(\O_\AA)}$. 

However, the commutative algebra structure on $\cI_{/\SL_2(\O_\AA)}$ allows us to make this canonical. Given an unramified spin structure, we can \emph{canonically} identify functions $\phi\in\cI_{/\SL_2(\O_\AA)}$ with functions on the unramified automorphic spectrum. In the most abstract terms, one can think of the algebraic spectrum $\Spec{\cI_{/\SL_2(\O_\AA)}}$ as a space, and of elements of $\cI_{/\SL_2(\O_\AA)}$ as functions on it. With this identification, the multiplication of $\cI_{/\SL_2(\O_\AA)}$ corresponds to point-wise multiplication on the spectrum, and is thus a convolution product.

This multiplicative structure will be described in Appendix~\ref{app:spin_struct}.

\begin{remark}
    One can speculate about the relationship of our construction with the geometric Langlands conjectures. Let $C$ be a smooth proper curve over $\CC$. Then the geometric Langlands conjectures predict the equivalence of a certain category of D-modules on $\Bun_{\GL_2}(C)$ with a certain category of sheaves of modules over $\mathrm{LocSys}_{\widecheck{\GL_2}}(C)$. Now, the category of D-modules over $\Bun_{\GL_2}(C)$ carries a symmetric monoidal structure, corresponding after de-categorification to the point-wise multiplication of unramified automorphic functions.
    
    However, one can speculate about giving this category a symmetric monoidal structure inherited from the other side of this equivalence, e.g. from the $!$-tensor product of $\mathrm{LocSys}_{\widecheck{\GL_2}}(C)$. This should correspond to a ``spectral'' product, or a convolution product, on unramified automorphic forms after de-categorification. The author believes that it would be interesting to attempt to relate this structure to the convolution product on $\cI_{/\SL_2(\O_\AA)}$ described above.
\end{remark}

\part{Symmetric Monoidal Structure} \label{part:monoidal_structure}

\section{Godement--Jacquet Vs. Jacquet--Langlands} \label{sect:GJ_vs_JL}

\subsection{Introduction}

Let $F$ be a non-Archimedean local field with $\chr(F)\neq 2$. In this section, we will attempt to give a new point of view on the equivalence between the Godement--Jacquet $L$-functions of representations of $G=\GL_2(F)$ and the Jacquet--Langlands $L$-functions of the same representations.

Consider the $G$-bi-module $S(\GL_2(F))$ of smooth and compactly supported functions on $\GL_2(F)$, and fix the standard Haar measure $\dtimes{g}$ on $\GL_2(F)$ such that the volume of a maximal compact subgroup is $1$. The spectral decomposition of this bi-module is well-known; for a pair of irreducible representations $(\pi,V)$ and $(\pi',V')$, we have
\[
    \tilde{V}\otimes_G S(\GL_2(F))\otimes_G V'\xrightarrow{\sim}\CC,
\]
if $V\cong V'$, and
\[
    \tilde{V}\otimes_G S(\GL_2(F))\otimes_G V'=0
\]
otherwise, where $\tilde{V}$ is the contragradient of $V$, considered as a right $G$-module.

In other words, every $G$-bi-module $V\otimes_\CC\tilde{V}$ appears ``once'' in the spectral decomposition of $S(\GL_2(F))$. To truly specify this decomposition, one needs to also discuss the dependence of this decomposition on the continuous part of the spectrum, but let us ignore this for now.

A question we can ask, then, is what is the spectral decomposition of the space $S(\M_2(F))$. The answer to this question should be fairly deep, as it relates to the Godement--Jacquet construction of $L$-functions. The simplest aspect of this we can ask about is to compute the spaces
\[
    \tilde{V}\otimes_G S(\M_2(F))\otimes_G V,
\]
with $(\pi,V)$ irreducible. Here, the Schwartz space $S(\M_2(F))$ is the space of smooth and compactly supported functions on $\M_2(F)$. We restrict to the case where $(\pi,V)$ is generic, i.e. has a Kirillov model. Note that we are also only looking at the diagonal part of the spectral decomposition, which we justify by observing that the two actions of the Bernstein center of $G$ on $S(\M_2(F))$ coincide (this follows because the two actions of the center on $S(G)$ coincide, as well as the fact that $S(\M_2(F))$ embeds in the contragradient of $S(G)$). This means that off-diagonal components
\[
    \tilde{V}\otimes_G S(\M_2(F))\otimes_G V',
\]
with $V'\neq V$ generic are $0$.

The space $\tilde{V}\otimes_G S(\M_2(F))\otimes_G V$ turns out to be a one-dimensional vector space. However, this answer in not too meaningful. We would like to say \emph{which} one-dimensional vector space this is. In order to be able to make interesting statements about it, we should give it some more structure. So, consider instead the space $Y=S(\M_2(F)\times F^\times)$. We let $G$ act on $\M_2(F)\times F^\times$ from both sides as:
\begin{equation} \label{eq:action_on_matrices}
    h\cdot(g,y)\cdot h'=\left(hg h',\,y\cdot\det(h h')^{-1}\right),
\end{equation}
which turns $Y=S(\M_2(F)\times F^\times)$ into a $G$-bi-module. Therefore, we want to compute the space
\[
    \tilde{V}\otimes_G Y\otimes_G V,
\]
which now has an extra structure of $F^\times$-module, via the action of $F^\times$ on the second term in the product $\M_2(F)\times F^\times$. The desired space $\tilde{V}\otimes_G S(\M_2(F))\otimes_G V$ is canonically the space of co-invariants of $\tilde{V}\otimes_G Y\otimes_G V$ under this $F^\times$-action.

Spectrally, the operation sending the $G$-bi-module $S(\M_2(F))$ to the $F^\times$-module $\tilde{V}\otimes_G S(\M_2(F)\times F^\times)\otimes_G V$ can be described as follows. Given a character $\chi\co F^\times\ra\CC^\times$, the component of $\tilde{V}\otimes_G S(\M_2(F)\times F^\times)\otimes_G V$ at $\chi$ is the same as the component of $S(\M_2(F))$ at the exterior product of $\pi\times(\chi\circ\det)$ with itself.

In order to have an idea of the answer, we can start by considering the $G\times G$-sub-module $Y^\circ=S(\GL_2(F)\times F^\times)$ of $Y$. It is easy to see that there is a canonical isomorphism
\begin{equation} \label{eq:spectrum_of_GL_2}
    \mu\co\tilde{V}\otimes_G Y^\circ\otimes_G V \xrightarrow{\sim}S(F^\times)
\end{equation}
given by the formula
\[
    \tilde{v}\otimes\Psi\otimes v \mapsto \int_{\GL_2(F)}\left<\tilde{v},\pi(g)v\right>\cdot\Psi\left(g,y\det(g)^{-1}\right)\dtimes{g}.
\]
Since $\M_2(F)$ is some completion of $\GL_2(F)$, we can expect $\tilde{V}\otimes_G Y\otimes_G V$ to be some extension of $S(F^\times)$.

\subsubsection{Statement of the Main Result}

We want to decompose the $G\times G$-module $Y$. It turns out that we can do this by introducing a \emph{third} action of $G$ on this space, coming from the Weil representation.

Fix a non-trivial additive character $e\co F\ra\CC^\times$, and let $\theta$ be the corresponding character on the subgroup $U=U_2(F)\subseteq\GL_2(F)$ of upper triangular unipotent matrices, defined by:
\begin{align*}
    \theta\co U & \ra\CC^\times \\
    \theta\left(\begin{pmatrix}1 & u \\ & 1\end{pmatrix}\right) & =e(u).
\end{align*}
For any generic irreducible $G$-module $V$, we denote by $V(n)$ the twist of $V$ by $\abs{\det(\cdot)}^{n}$.

\begin{recollection} \label{recollection:kirillov}
    Recall that for a generic irreducible $G$-module $V$, the \emph{Kirillov model} $\mathcal{K}(V)$ of $V$ is defined as follows.
    
    Let $P=\left\{\begin{pmatrix}* & * \\ & 1\end{pmatrix}\right\}\subseteq G$ be the mirabolic subgroup. Let $\Phi^-\co\Mod(P)\ra\Vect$ be the functor that takes a $P$-module to its $\theta$-equivariant quotient, for the non-trivial character $\theta$ on the unipotent radical of $P$ (see \cite{derivatives_of_p_adic_repsI} for this notation). Let $\Phi^+\co\Vect\ra\Mod(P)$ be its left adjoint, given by $\theta$-equivariant induction with compact support (see Proposition~3.2 of \cite{derivatives_of_p_adic_repsI}).
    
    It is well known that the space $\Phi^-\tilde{V}$ is one dimensional. A choice of vector $\CC\ra\Phi^-\tilde{V}$ gives, by adjunction, a map of $P$-modules:
    \[
        S(F^\times)\cong\Phi^+\CC\ra\tilde{V}.
    \]
    Taking the dual map, we obtain a morphism of $P$-modules:
    \[
    V\ra\widetilde{S(F^\times)},
    \]
    where $\widetilde{S(F^\times)}$ is identified with the space of smooth functions on $F^\times$, and $P$ acts on $\widetilde{S(F^\times)}$ by:
    \[
        \begin{pmatrix}a & b \\ & 1\end{pmatrix}\cdot f(y)=e(by)f(ay).
    \]
    
    The map $V\ra\widetilde{S(F^\times)}$ is the unique non-zero map of $P$-modules from $V$ to $\widetilde{S(F^\times)}$, up to scalar. Such maps are injective, and have the same image. This image is called the Kirillov model $\mathcal{K}(V)$ of $V$, and it acquires a canonical $G$-action from $V$ \cite[Theorem~6.7.2]{aut_reps_book_I}.
    
    There is a related notion, called the \emph{Whittaker model} of $V$. This model is obtained by taking the map $S(F^\times)\ra\tilde{V}$ of $P$-modules above, and extending it to a map of $G$-modules:
    \[
    \one_\Ydown\to\tilde{V}
    \]
    by letting $\one_\Ydown$ be the induction with compact support of $S(F^\times)$ from $P$ to $G$ (see also Construction~\ref{const:Y_unit}). The space space $\one_\Ydown$ is called the \emph{Whittaker space}, and the image of the resulting dual map $V\to\widetilde{\one_\Ydown}$ is called the Whittaker model of $V$. Our notation for the Whittaker space is non-standard.
    
\end{recollection}

\begin{theorem} \label{thm:middle_action}
    There exists an additional canonical $G$-action $\tau$ on $Y$. This action is uniquely determined by the following properties:
    \begin{enumerate}
        \item The action $\tau$ commutes with the existing $G$-actions on $Y$ from the left and the right.
        \item \label{item:V_is_idemp} For every generic irreducible $(\pi,V)$, there is an isomorphism of $G$-modules
        \[
            \nu\co\tilde{V}\otimes_G Y\otimes_G V\xrightarrow{\sim}\mathcal{K}(V(-1))(1).
        \]
        \item \label{item:compatibility_mu_nu} The isomorphism $\nu$ above fits into the commutative diagram
        \begin{equation} \label{eq:compatible_M_2_GL_2} \xymatrix{
            \tilde{V}\otimes_G Y^\circ\otimes_G V \ar[d] \ar[r]^-\sim_-\mu & S(F^\times) \ar[d] \\
            \tilde{V}\otimes_G Y\otimes_G V \ar[r]^-\sim_-\nu & \mathcal{K}(V(-1))(1),
        }\end{equation}
        where $\mu$ is as in Equation~\eqref{eq:spectrum_of_GL_2}, the left vertical map is induced from the natural inclusion $Y^\circ\subseteq Y$, and the right vertical map is the inclusion of $S(F^\times)$ into the Kirillov model $\mathcal{K}(V(-1))$ of $V(-1)$.
    \end{enumerate}
\end{theorem}

\begin{remark}
    Note that in Item~\ref{item:V_is_idemp} above, the right hand side $\mathcal{K}(V(-1))(1)$ is isomorphic to $V$ as a representation of $G$, but ``rigidified'' via the use of Kirillov models (in the sense that while an irreducible automorphic representation is only unique up to automorphism, the Kirillov model is a canonical realization of that representation).
\end{remark}

\begin{remark}
    Note that the right vertical map appearing in Diagram~\eqref{eq:compatible_M_2_GL_2} is not $P$-equivariant. Rather, this map is twisted-$P$-equivariant, with respect to the character $\abs{\det(\cdot)}$.
\end{remark}

We refer to the $G$-action on $Y$ introduced in Theorem~\ref{thm:middle_action} as the \emph{middle} $G$-action (in contrast with the left or right $G$-actions). The isomorphism $\nu$ in Theorem~\ref{thm:middle_action} is our answer to the questions posed above. That is, $\nu$ defines an isomorphism:
\[
\tilde{V}\otimes_G S(\M_2(F))\otimes_G V\cong\mathcal{K}(V(-1))(1)_{/F^\times}.
\]

\subsubsection{Compatibility With \texorpdfstring{$L$}{L}-Functions}

The isomorphism $\nu$ above has one more remarkable property: it intertwines Godement--Jacquet and Jacquet--Langlands zeta integrals (and thus shows that these two methods give the same $L$-function). Explicitly, we make the following claim.

\begin{proposition}
    Let $\Psi\otimes f\in S(\M_2(F))\otimes_\CC S(F^\times)=Y$ be the function $\Psi(g)f(y)$, and let
    \[
        j\co\mathcal{K}(V(-1))(1)\xrightarrow{\sim}\mathcal{K}(V)
    \]
    be the canonical isomorphism. Then
    \[
        Z_\text{JL}\left(j\circ\nu(\beta\otimes\Psi\otimes f),s\right)=Z_\text{GJ}\left(\beta,\Psi,s\right)\cdot\int_{F^\times} f(y)\abs{y}^{s+\frac{1}{2}}\dtimes{y},
    \]
    where
    \begin{align*}
        Z_\text{JL}\left(w,s\right) & =\int_{F^\times} w(y)\abs{y}^{s-\frac{1}{2}}\dtimes{y}, \\
        Z_\text{GJ}\left(\beta,\Psi,s\right) & =\int_{\GL_2(F)} \beta(g)\Psi(g)\abs{\det(g)}^{s+\frac{1}{2}}\dtimes{g} \\
    \end{align*}
    are the Jacquet--Langlands and Godement--Jacquet zeta integrals, respectively.
\end{proposition}
This will follow immediately from Remark~\ref{remark:zeta_for_Y} below.

\begin{remark}
    With the language and tools developed in Section~\ref{sect:T_tensor}, we will be able to re-formulate Theorem~\ref{thm:middle_action} in a much cleaner way, as well as give a slightly different proof.
    
    Specifically, it will turn out that we can use the space $Y$ to define a symmetric monoidal structure $\oY$ on the category of smooth $G$-modules. The essence of Theorem~\ref{thm:middle_action} is that generic irreducible representations are idempotent with respect to this monoidal structure. This will follow at once from the general principle that in an abelian symmetric monoidal category whose tensor product is right exact, surjections from the unit $\one_\Ydown\twoheadrightarrow V$ turn $V$ into an idempotent. This is essentially the same argument showing that for a commutative algebra $A$, module quotients of $A$ are also algebra quotients of $A$. See Example~\ref{example:irr_reps_are_algs} for details.
\end{remark}

\subsection{The Weil Representation}

In this subsection, we will construct the middle action on $Y$, which we will use to prove Theorem~\ref{thm:middle_action}. This is done by turning $S(\M_2(F))$ into a Weil representation for an appropriate metaplectic group. The action of the metaplectic group is then used to construct the hidden action on $Y$ directly. Once this action is constructed, proving its properties is relatively straightforward. After proving Theorem~\ref{thm:middle_action}, we will give a few additional remarks about the middle action and its properties.

We begin by constructing the symplectic space which will give us the appropriate representation. The idea is to find a space such that $\M_2(F)$ is a Lagrangian subspace of it.

Recall that the pairing $(m,m')\mapsto\left<m,m'\right>=\tr(m)\tr(m')-\tr(mm')$ is the polarization (or multi-linearization) of the quadratic map $m\mapsto\det(m)$ on $\M_2(F)$.
\begin{construction} \label{const:G3_into_symp}
    Consider the vector space $U=\M_2(F)$, and let $W$ be a two dimensional symplectic vector space. We turn $U\otimes W$ into a symplectic vector space via
    \[
        \left<m\otimes v,m'\otimes v'\right>=(v\wedge v')\cdot\left<m,m'\right>.
    \]
    
    In particular, we get a map
    \[
        G^{3,\det=1}=(\GL_2(F)\times\GL_2(F)\times\GL_2(F))^1\ra\Symp(U\otimes W)
    \]
    of the group
    \[
        G^{3,\det=1}=\left\{(g_1,g_2,g_3)\in G^3\suchthat\det(g_1 g_2 g_3)=1\right\}
    \]
    into the group of symplectic automorphisms of $U\otimes W$. This embedding is defined by:
    \[
        (g_1,g_2,g_3)\cdot(m\otimes v)=g_1 m g_3^{T}\otimes g_2 v.
    \]
\end{construction}

\begin{remark} \label{remark:three_actions_coincide}
    The construction of the space $U\otimes W$ above obfuscates the fact that the three terms in $G^{3,\det=1}$ act on $U\otimes W$ symmetrically. Indeed, as a vector space, $U\cong W\otimes W^\vee$, and therefore one can exchange the factors of $W$ in the product $U\otimes W\cong W\otimes W^\vee\otimes W$. This is easily seen to leave the symplectic form invariant.
\end{remark}

\begin{construction}
    The map $G^{3,\det=1}\ra\Symp(U\otimes W)$ of Construction~\ref{const:G3_into_symp} can be extended to a map:
    \[
        S_3\ltimes G^{3,\det=1}\ra\Symp(U\otimes W)
    \]
    where $S_3$ acts by interchanging the order of the factors. In particular, we fix the convention that the transposition $(1,3)\in S_3$ acts by:
    \[
    (1,3)\cdot(m\otimes v)=m^{T}\otimes v.
    \]
\end{construction}

\begin{construction}
    Consider the Lagrangian subspace $\M_2(F)=\M_2(F)\otimes e_2\subseteq U\otimes W$, where $e_1,e_2\in W$ is a standard basis. Then by the Schr\"odinger model of the Weil representation of $\Mp(U\otimes W)$ on $S(\M_2(F))$ corresponding to the character $e\co F\ra\CC^\times$, we get an action of the metaplectic group
    \[
        \Mp(U\otimes W)
    \]
    on $S(\M_2(F))$.
\end{construction}

\begin{proposition} \label{prop:Y_three_actions_coincide}
    There is a unique lift
    \[\xymatrix{
        S_3\ltimes G^{3,\det=1} \ar[r] \ar@{-->}[rd] & \Symp(U\otimes W) \\
        & \Mp(U\otimes W) \ar[u]
    }\]
    satisfying that the transposition $(1,3)\in S_3$ of the left and right actions acts on $S(\M_2(F))$ via the transposition $g\mapsto g^T$ of $\M_2(F)$.
\end{proposition}
For the sake of clarity of the exposition, we postpone the proof of this statement to the end of this subsection. Our proof will be constructive, and we will give an explicit description of this lift.

We conclude that the group $S_3\ltimes G^{3,\det=1}$ acts on $S(\M_2(F))$.

Finally, we can use this Weil representation to construct an additional $G$-action on $Y$.
\begin{construction} \label{const:middle_action}
    We use induction with compact support to extend the action of $G^{3,\det=1}$ on $S(\M_2(F))$ to an action of $G^3$ on $Y$. Specifically, we identify
    \[
        Y=S(\M_2(F)\times F^\times)=\cInd_{S_3\ltimes G^{3,\det=1}}^{S_3\ltimes G^3}S(\M_2(F))
    \]
    via the section
    \[
        G^{3,\det=1}\backslash G^3=F^\times\ra G^3
    \]
    given by $y\mapsto\left(1,\begin{pmatrix}y & \\ & 1\end{pmatrix},1\right)$.
    
    Explicitly, if $\Psi(m,y)=\Phi(m)f(y)\in S(\M_2(F)\times F^\times)$ for $\Phi\in S(\M_2(F))$ and $f\in S(F^\times)$, then we define:
    \begin{multline} \label{eq:induced_action_on_Y}
        ((g_1,g_2,g_3)\cdot\Psi)(m,y)= \\
        =\abs{\det(g_1 g_2 g_3)}\cdot\left(\left(g_1,\begin{pmatrix}y & \\ & 1\end{pmatrix}g_2\begin{pmatrix}y^{-1}\det(g_1 g_2 g_3)^{-1} & \\ & 1\end{pmatrix},g_3\right)\cdot\Phi\right)(m)\times \\
        \times f(y\cdot\det(g_1 g_2 g_3)),
    \end{multline}
    where the action of $\left(g_1,\begin{pmatrix}y & \\ & 1\end{pmatrix}g_2\begin{pmatrix}y^{-1}\det(g_1 g_2 g_3)^{-1} & \\ & 1\end{pmatrix},g_3\right)$ on $\Phi$ is the Weil representation action.
    
    We refer to the resulting action of the middle copy of $G=\GL_2(F)$ on $Y$ as the \emph{middle action}.
\end{construction}

Recall that $g\mapsto g^{-T}$ is the Cartan involution on $\GL_2(F)$, given by the inverse of the transposition map: $g^{-T}=(g^T)^{-1}$.
\begin{example} \label{example:left_and_right_actions}
    We can explicitly write down the action of $(g_1,1,g_3)$ on $Y$:
    \[
        (g_1,1,g_3)\cdot\Psi(g,y)=\Psi\left(g_1^{-1}g g_3^{-T},y\cdot\det(g_1 g_3^{T})\right).
    \]
    This follows from the explicit construction in Equations~\eqref{eq:left_right_metaplectic_actions} and~\eqref{eq:S_2_F_3_metaplectic_action} in the proof of Proposition~\ref{prop:Y_three_actions_coincide} below.
\end{example}

\begin{example}
    As in Remark~\ref{remark:three_actions_coincide}, there is a hidden symmetry between the left, right and middle $G$-actions on $Y$. This follows because $Y$ carries an action of $S_3\ltimes G^3$. In particular, the action of the the transposition $(1,3)\in S_3$ on $Y$ is given by:
    \[
    (1,3)\cdot\Psi(g,y)=\Psi(g^T,y),
    \]
    as can be seen from Equation~\eqref{eq:transpose_metaplectic_action} in the proof of Proposition~\ref{prop:Y_three_actions_coincide} below.
\end{example}

\begin{remark}
    In order to be consistent with the notation of Theorem~\ref{thm:middle_action}, we need to be able to write expressions of the form:
    \[
        \widetilde{V}\otimes_G Y\otimes_G V.
    \]
    Recall that we are considering $\widetilde{V}$ to be a right-module. Therefore, we need to turn one of the $G$-actions on $Y$ into a right action. In order to be consistent with Equation~\eqref{eq:action_on_matrices}, we will occasionally use the transposition $g\mapsto g^T$ to turn the third $G$-action into a right-action whenever it is necessary. See also Definition~\ref{def:directions_for_G_modules} below.
\end{remark}

We finish this subsection by proving Proposition~\ref{prop:Y_three_actions_coincide}.
\begin{proof}[Proof of Proposition~\ref{prop:Y_three_actions_coincide}]
    Let us begin by showing that there is at most one metaplectic lift for the map $S_3\ltimes G^{3,\det=1}\ra\Symp(U\otimes W)$ compatible with the specified action of the metaplectic lift of $(1,3)$. This follows because the abelianization of $S_3\ltimes G^{3,\det=1}$ is $\ZZ/2$, and is generated by the image of the transposition $(1,3)$, due to the fact that any two lifts to $\Mp(U\otimes W)$ differ by a character. Also recall that the Weil representation is faithful for the metaplectic group, so that describing lifts of elements of $S_3\ltimes G^{3,\det=1}$ is the same as describing their actions on $S(\M_2(F))$.
    
    Let us show the existence of this lift. We will do so in several steps. Let us first show that the subgroup $\SL_2(F)^3\subseteq S_3\ltimes G^{3,\det=1}$ admits a unique lift. We will then handle the rest of the group. Indeed, the lift of each factor $\SL_2(F)$ is unique because its abelianization is trivial. Note that the subgroup $\SL_2(F)\times\{1\}\times\SL_2(F)$ consisting of two of the copies of $\SL_2(F)$ admits an explicit metaplectic lift by the following action on $S(\M_2(F))$:
    \begin{equation} \label{eq:left_right_metaplectic_actions}
        (g_1,1,g_3)\cdot \Psi(m)=\Psi(g_1^{-1}m g_3^{-T}).
    \end{equation}
    Because the pairs $\SL_2(F)\times\SL_2(F)\times\{1\}$ and $\{1\}\times\SL_2(F)\times\SL_2(F)$ are conjugate to it, these pairs admit metaplectic lifts as well, and all of those lifts agree on each factor $\SL_2(F)$ separately by uniqueness. This shows that there is a metaplectic lift for $\SL_2(F)\times\SL_2(F)\times\SL_2(F)$. For the sake of being explicit, we note that the action of the middle copy of $\SL_2(F)$ is given by:
    \begin{align}
        \left(1,\begin{pmatrix}1 & b \\ & 1\end{pmatrix},1\right)\cdot \Psi(m) & =e\left(b\det(m)\right)\cdot\Psi(m) \label{eq:middle_unipotent_metaplectic_action} \\
        \left(1,\begin{pmatrix}t & \\ & t^{-1}\end{pmatrix},1\right)\cdot \Psi(m) & =\abs{t}^2\cdot\Psi(t\cdot m) \\
        \left(1,\begin{pmatrix}& -1 \\ 1 &\end{pmatrix},1\right)\cdot \Psi(m) & =\int\limits_{\M_2(F)}\Psi(n)e(-\left<m,n\right>)\d{n}, \label{eq:middle_weyl_metaplectic_action}
    \end{align}
    where the measure $\d{n}$ on $\M_2(F)$ is normalized such that $e(\left<\cdot,\cdot\right>)$ is a self-dual character.
    
    Now, we observe that there is a split short exact sequence:
    \[\xymatrix{
    1 \ar[r] & \SL_2(F)^3 \ar[r] & S_3\ltimes G^{3,\det=1} \ar[r] & S_3\ltimes (F^\times)^{3,\Pi=1} \ar[r] & 1
    }\]
    where $(F^\times)^{3,\Pi=1}=\left\{(a,b,c)\in F^\times\suchthat abc=1\right\}$. We think of $(F^\times)^{3,\Pi=1}$ as embedded in $G^{3,\det=1}$ via the splitting:
    \[
    (a,b,c)\mapsto\left(\begin{pmatrix}a & \\ & 1\end{pmatrix},\begin{pmatrix}b & \\ & 1\end{pmatrix},\begin{pmatrix}c & \\ & 1\end{pmatrix}\right).
    \]
    Thus, it remains to show that there is a metaplectic lift for $S_3\ltimes (F^\times)^{3,\Pi=1}$, as it will act by conjugation on $\SL_2(F)^3$ correctly by uniqueness.
    
    We construct the lift for $S_3\ltimes (F^\times)^{3,\Pi=1}$ in three steps. We will first construct a lift $S_2\ltimes (F^\times)^{3,\Pi=1}$ for the copy of $S_2\subseteq S_3$ generated by the permutation $\tau=(1,3)$. Then, we will construct a lift for a permutation $\sigma\in S_3$ of order $3$. Finally, we will show that all required relations between $\sigma$ and $S_2\ltimes (F^\times)^{3,\Pi=1}$ must hold.
    
    We lift $S_2\ltimes (F^\times)^{3,\Pi=1}$ explicitly by defining:
    \begin{align} \label{eq:S_2_F_3_metaplectic_action}
        \left(\begin{pmatrix}a & \\ & 1\end{pmatrix},\begin{pmatrix}b & \\ & 1\end{pmatrix},\begin{pmatrix}c & \\ & 1\end{pmatrix}\right)\cdot\Psi(m) & =\abs{b}\cdot\Psi\left(\begin{pmatrix}a^{-1} & \\ & 1\end{pmatrix}m\begin{pmatrix}c^{-1} & \\ & 1\end{pmatrix}\right) \\
        \tau\cdot\Psi(m) & =\Psi\left(m^{T}\right) \label{eq:transpose_metaplectic_action}
    \end{align}
    for $\Psi\in S(\M_2(F))$, and $\tau=(1,3)\in S_3$ being the transposition of the first and third actions.
    
    Let $\sigma\in S_3$ be the permutation $\sigma(1)=2,\sigma(2)=3,\sigma(3)=1$ of order $3$. Observe that there is a unique lift for sigma in $\Mp(U\otimes W)$ which respects the relation:
    \[
    \tau\sigma\tau^{-1}=\sigma^2.
    \]
    Note that because $\tau$ has order $2$, such a lift automatically respects all relations of $S_3$. For the sake of being explicit, we note that this lift is given by:
    \begin{equation} \label{eq:ord_3_permutation_metaplectic_action}
        \sigma\cdot\Psi\left(\begin{pmatrix}m_{00} & m_{01} \\ m_{10} & m_{11}\end{pmatrix}\right)=\int\limits_{F\times F}e(\beta\cdot m_{00}-\alpha\cdot m_{10})\Psi\left(\begin{pmatrix}\alpha & \beta \\ m_{01} & m_{11}\end{pmatrix}\right)\d{\alpha}\d{\beta}.
    \end{equation}
    
    We have now chosen lifts for both $S_2\ltimes (F^\times)^{3,\Pi=1}$ and $\sigma$. Thus, to show that $S_3\ltimes (F^\times)^{3,\Pi=1}$ admits a metaplectic lift and finish the proof, it suffices to check that the action of $S_3$ on the metaplectic lift of $(F^\times)^{3,\Pi=1}$ by conjugation is the correct one. This can be verified explicitly using Equations~\eqref{eq:S_2_F_3_metaplectic_action} and~\eqref{eq:ord_3_permutation_metaplectic_action}
\end{proof}

\subsection{Proof of Theorem~\ref{thm:middle_action}}

Having constructed the action we seek, we can now start proving Theorem~\ref{thm:middle_action}. Our first step is to restrict the set of representations that can appear in $\tilde{V}\otimes_G Y\otimes_G V$, which is done via the following lemma.

\begin{lemma} \label{lemma:center_is_compatible}
    Let $Y$ be considered as a $G^3$-module via the left, middle and right actions respectively. Then the resulting three actions of the Bernstein centers of each copy of $G$ on $Y$ identify.
\end{lemma}
\begin{proof}
    Observe that the statement is true for the left and right translation actions of $G$ on $S(\M_2(F))$ because the $G\times G$-equivariant pairing:
    \begin{align*}
        S(\M_2(F))\otimes S(G) & \ra\CC \\
        \Psi\otimes f & \mapsto \int_G \Psi(g)f(g)\dtimes{g}
    \end{align*}
    is non-degenerate and because the Bernstein center of $G$ acts on $S(G)$ the same way from both sides. As a result, the left and right actions of the center of $G$ on $Y$ coincide. By symmetry (i.e., by conjugation by a permutation from $S_3$), this also applies to the middle action.
\end{proof}

Recall the functor $\Phi^-\co\Mod(P)\ra\Vect$ and its left adjoint $\Phi^+\co\Vect\ra\Mod(P)$ from recollection~\ref{recollection:kirillov}.
\begin{lemma} \label{lemma:mirabolic_rest}
    The map
    \[
        \Phi^+\Phi^-Y\ra Y
    \]
    is an embedding with image $Y^\circ\subseteq Y$, where the functor $\Phi^+\Phi^-$ is taken with respect to the middle action.
\end{lemma}
\begin{proof}
    By Equation~\eqref{eq:middle_unipotent_metaplectic_action} and the definition in Equation~\eqref{eq:induced_action_on_Y}, the middle action of the matrix $\begin{pmatrix} a & b \\ & 1\end{pmatrix}$ on the element $\Psi\in Y$ is:
    \begin{equation} \label{eq:mirabolic_middle_action}
        \begin{pmatrix} a & b \\ & 1\end{pmatrix}\cdot\Psi(g,y)=\abs{a}\cdot e\left(b\det(g)y\right)\cdot\Psi(g,ay).
    \end{equation}
    
    By the exact sequence of part~(e) of Proposition~3.2 of \cite{derivatives_of_p_adic_repsI}, it follows that the image in question consists of the functions supported away from the set $\{\det(g)y=0\}$.
\end{proof}

\begin{proof}[Proof of Theorem~\ref{thm:middle_action}]
    We consider the middle action of $G$ on $Y$, given by Construction~\ref{const:middle_action}. That it commutes with the left and right $G$-actions is by construction. The existence of the isomorphism $\nu$ of Theorem~\ref{thm:middle_action} is proven as follows.
    
    First, we suppose that $(\pi,V)$ is either an irreducible principal series representation or a supercuspidal representation. By Lemma~\ref{lemma:center_is_compatible}, the space
    \[
        \tilde{V}\otimes_G Y\otimes_G V
    \]
    is $V$-isotypic, because such irreducible representations are determined by their central characters.
    Our goal is to count the number of copies of $V$ appearing in it, and show that it is $1$. By the assumptions on $V$, we have that $\Phi^-V$ is one-dimensional. Therefore, it is enough to test the dimension of
    \[
        \Phi^-\left(\tilde{V}\otimes_G Y\otimes_G V\right).
    \]
    However, this is the same as:
    \[
        \Phi^-\Phi^+\Phi^-\left(\tilde{V}\otimes_G Y\otimes_G V\right),
    \]
    which by Lemma~\ref{lemma:mirabolic_rest} is the same as evaluating
    \[
        \Phi^-\left(\tilde{V}\otimes_G Y^\circ\otimes_G V\right)\xrightarrow{\sim}\Phi^- S(F^\times)=\CC.
    \]
    Note that the first isomorphism is given by $\Phi^-(\mu)$ for $\mu$ as in Equation~\eqref{eq:spectrum_of_GL_2}, and it follows from Equation~\eqref{eq:mirabolic_middle_action} that $\mu$ is an isomorphism of $P$-modules. This shows that the number of copies of $V$ appearing in the space above is indeed $1$, and shows that the diagram in Item~\ref{item:compatibility_mu_nu} is commutative.
    
    The above argument does not quite work in the case that $(\pi,V)$ is a Steinberg representation, because the Bernstein center is not strong enough to distinguish $V$ from the trivial representation. However, in this case, it is sufficient to verify that the image of the Jacquet functor on $\tilde{V}\otimes_G Y\otimes_G V$ is one-dimensional (because we already know that $\Phi^-\left(\tilde{V}\otimes_G Y\otimes_G V\right)$ is one-dimensional by the argument above). By Equation~\eqref{eq:mirabolic_middle_action}, it follows that the Jacquet functor on $\tilde{V}\otimes_G Y\otimes_G V$ is given by:
    \[
        \tilde{V}\otimes_G S(\M_2^{\det=0}(F)\times F^\times)\otimes_G V.
    \]
    Hence, we need to prove that this space is one-dimensional. This follows immediately from Lemma~\ref{lemma:mirabolic_jacquet} below (note that $V$ is a Steinberg representation, and therefore generic, meaning that it admits no $\SL_2(F)$-invariant linear functionals, as the lemma requires).
\end{proof}

Identify the torus $\left\{\begin{pmatrix}a & \\ & d\end{pmatrix}\right\}\subseteq G$ with the group $F^\times\times F^\times$. Let
\[
    \Jac\co\Mod(G)\ra\Mod(F^\times\times F^\times)
\]
be the Jacquet functor (i.e., parabolic restriction) with respect to the parabolic subgroup
\[
    P'=\left\{\begin{pmatrix}a & b \\ & d\end{pmatrix}\right\}\subseteq G.
\]
\begin{lemma} \label{lemma:mirabolic_jacquet}
    Let $(\pi,V)$ be a representation of $G$, such that $V$ and $\tilde{V}$ admit no $\SL_2(F)$-invariant linear functionals. Then
    \[
        \tilde{V}\otimes_G S(\M_2^{\det=0}(F)\times F^\times)\otimes_G V=\Jac(\tilde{V})\otimes_{F^\times\times F^\times}\Jac(V).
    \]
\end{lemma}
\begin{proof}
    We use the decomposition of $S(\M_2^{\det=0}(F))$ via an exact sequence:
    \[\xymatrix{
        0 \ar[r] & S(P\backslash G)\otimes_{F^\times}S(P\backslash G) \ar[r] & S(\M_2^{\det=0}(F)) \ar[r]^-{f\mapsto f(0)} & \one_G \ar[r] & 0,
    }\]
    where $\one_G$ is the trivial representation. This uses an identification of the quotient of
    \[
        P\backslash G\times P\backslash G
    \]
    by the diagonal left action of $F^\times$ with the space of rank-one matrices.
    
    If we tensor this sequence by $S(F^\times)$, we obtain the decomposition:
    \begin{equation*}\xymatrix{
        0 \ar[r] & S(U\backslash G)\otimes_{F^\times\times F^\times}S(U\backslash G) \ar[r] & \\
        \ar[r] & S(\M_2^{\det=0}(F)\times F^\times) \ar[r] & S(F^\times) \ar[r] & 0,
    }\end{equation*}
    where $U\subseteq G$ is the unipotent radical. Tensoring by $\tilde{V}$ and $V$ from the left and right, we get a morphism:
    \begin{equation} \label{eq:map_jacquet_to_mat_det_0} \xymatrix{
        \Jac(\tilde{V})\otimes_{F^\times\times F^\times}\Jac(V) \ar[r] & \tilde{V}\otimes_G S(\M_2^{\det=0}(F)\times F^\times)\otimes_G V,
    }\end{equation}
    and we want to show that it is an isomorphism.
    
    The fact that \eqref{eq:map_jacquet_to_mat_det_0} is an isomorphism will follow because our assumption about $\SL_2(F)$-invariant linear functionals implies that
    \begin{equation} \label{eq:no_SL2_functionals}
        S(F^\times)\otimes_G V=S(F^\times)\otimes_G \tilde{V}=0.
    \end{equation}
    Indeed, surjectivity immediately follows from $\tilde{V}\otimes_G S(F^\times)\otimes_G V=0$ and because relative tensor products are right exact. Injectivity is trickier, but can be proven as follows. Essentially, it remains to prove that
    \[
        \Tor_1^{G\times G}(\tilde{V}\otimes V,S(F^\times))=0
    \]
    Plugging Equation~\eqref{eq:no_SL2_functionals} in to the spectral sequence
    \[
        \Tor_i^G(\tilde{V},\Tor_j^G(V,S(F^\times))) \implies \Tor_{i+j}^{G\times G}(\tilde{V}\otimes V,S(F^\times)),
    \]
    we see that the objects at $(i,j)=(0,0),(1,0),(0,1)$ are all $0$.
\end{proof}

\subsection{Properties of the Hidden Action}

Let us make a few closing remarks about the properties of the middle $G$-action on the space $Y$. These observations are easy to see directly, but we write them down explicitly because they will be useful later.

\begin{remark} \label{remark:zeta_for_Y}
    The consistency of the map
    \[
        \nu\co\tilde{V}\otimes_G Y\otimes_G V\ra \mathcal{K}(V(-1))(1)
    \]
    with the map $\mu$ allows us to write it explicitly. It is easy to see that $\nu$ is given by:
    \[
        \nu\left(\tilde{v}\otimes\Psi\otimes v\right)(y)=\int\limits_{\mathrm{M}_2(F)}\left<\tilde{v},\pi(g)v\right>\cdot\Psi\left(g,\det(g)^{-1}y\right)\dtimes{g}.
    \]
    That is, Theorem~\ref{thm:middle_action} says that the expression above, when considered as a function of $y$, belongs to the twist of the Kirillov space of $V(-1)$.
\end{remark}

\begin{remark}
    By Example~\ref{example:left_and_right_actions}, the subspace $Y^\circ\subseteq Y$ is closed under the left and right $G$-actions (and, in fact, Equation~\eqref{eq:mirabolic_middle_action} shows that it is also closed under the middle action of the mirabolic group $P$). However, one can show that $Y^\circ$ generates the entire space $Y$ under the middle action of $G$.
\end{remark}

\begin{remark}
    Very informally, Theorem~\ref{thm:middle_action} says that the spectral decomposition of $Y$ is given by an integral of the form
    \[
        Y\cong\int V\otimes\mathcal{K}(V(-1))(1)\otimes\tilde{V}
    \]
    over the irreducible representations $(\pi,V)$ of $G$.
\end{remark}

\begin{remark} \label{remark:middle_w_action}
    We can also write down the middle action of the matrix $\begin{pmatrix} & -1 \\ 1 & \end{pmatrix}$ on $Y$. Explicitly, by normalizing the Haar measure on $\M_2(F)$ so that $e(\left<-,-\right>)$ is self-dual, it is straightforward to use Equation~\eqref{eq:middle_weyl_metaplectic_action} to obtain:
    \[
        \begin{pmatrix} & -1 \\ 1 & \end{pmatrix}\cdot\Psi(g,y)=\abs{y}^2\int_{\M_2(F)}\Psi(h,y)\cdot e\bigg(-y\cdot\left<g,h\right>\bigg)\d{h}.
    \]
\end{remark}

\section{Symmetric Monoidal Structure on Modules of \texorpdfstring{$\GL(2)$}{GL(2)}} \label{sect:T_tensor}

\subsection{Introduction}

Let $F$ be a non-Archimedean local field with $\chr(F)\neq 2$, and fix the standard Haar measure on $\GL_2(F)$ such that the volume of a maximal compact subgroup is $1$. In Section~\ref{sect:GJ_vs_JL}, we introduced the object
\[
    Y=S(\M_2(F)\times F^\times).
\]
The space $Y$ has three commuting $G=\GL_2(F)$-actions. The first copy of $G$ acts by multiplication on the left on $\M_2(F)$ and by the determinant on $F^\times$, and the second copy of $G$ acts by multiplication on the right on $\M_2(F)$ and by the determinant on $F^\times$. The third action is harder to describe, and is related to the Weil representation.

The tri-module $Y$ appears to encode information about the relation between the Godement--Jacquet and Jacquet--Langlands $L$-functions. An interesting question to ask is how can we use $Y$ to learn more about the representation theory of $G$. Regardless, one might want to better understand the tri-module $Y$ and its properties. It turns out that a useful way of achieving both is by considering the tri-module $Y$ as encoding a functor
\[
    \oY\co\Mod(G)\times\Mod(G)\ra\Mod(G),
\]
sending pairs of modules to their tensor product with $Y$ over $G$, and letting $G$ act on the result via the third action:
\[
    V\oY V'=V\otimes_G Y \otimes_G V'.
\]

This point of view turns out to be surprisingly insightful: the functor $\oY$ turns out to be a part of a (unital) symmetric monoidal structure on the category $\Mod(G)$. In fact, this structure can be further enriched with the duality on $\Mod(G)$ (which interchanges left and right modules) into a structure called a \emph{commutative Frobenius algebra} in presentable categories in an appropriately categorified sense (see Remark~\ref{remark:oY_is_frob}).

The symmetric monoidal structure $\oY$ enjoys some interesting properties. For example, the main result of Section~\ref{sect:GJ_vs_JL} can be re-stated in terms of $\oY$: namely, Theorem~\ref{thm:middle_action} says that generic irreducible representations $V$ of $G$ satisfy $V\oY V\cong V$ (see Example~\ref{example:irr_reps_are_algs} for a more precise formulation). Moreover, in Section~\ref{sect:T_global}, we will show that the functor $\oY$ can be turned into a symmetric monoidal product on the category of global $\GL_2(\AA)$-modules, and that an appropriate space of global automorphic functions is an algebra under $\oY$.

\begin{remark} \label{remark:oY_is_like_comm}
    Intuitively, one can think of the product $\oY$ as a kind of ``commutative relative tensor product'' over $G$. Let us explain this idea by considering a commutative group $H$ and showing that its relative tensor product $\otimes_H$ satisfies similar properties to those of $\oY$.
    
    For commutative groups $H$, the relative tensor product $\otimes_{H}\co\Mod(H)\times\Mod(H)\ra\Mod(H)$ is defined as follows. For $M,N\in\Mod(H)$, the vector space $M\otimes N$ has a $H\times H$-action. Then the $H$-module $M\otimes_H N$ is given by taking the co-invariants of $M\otimes N$ by the action of the subgroup
    \[
        \{(h,h^{-1})\suchthat h\in H\}\subseteq H\times H.
    \]
    
    The relative tensor product $\otimes_{H}$ in the commutative case seems to satisfy many of the same properties as the bi-functor $\oY$ in the case of $\GL(2)$. For example, for an irreducible $H$-module $V$, there is an isomorphism $V\otimes_{H}V\cong V$. This is analogous to the above re-statement of Theorem~\ref{thm:middle_action}. Moreover, the space of functions on $H$ acquires an algebra structure with respect to this monoidal structure, which is analogous to the fact that the space of global automorphic functions has a multiplication map with respect to $\oY$.
    
    In this informal manner (e.g., having the tensor of generic irreducibles be themselves), the monoidal structure $\oY$ ``mimics'' the behaviour of the case of a commutative group.
\end{remark}

This section is dedicated to constructing all of this extra structure on top of the object $Y$, and is structured as follows. In Subsection~\ref{subsect:def_Y}, we rigorously introduce the bi-functor $\oY$. Unfortunately, at this point we do not yet know that it can be extended to a symmetric monoidal structure, as this requires introducing unitality and associativity data. In Subsection~\ref{subsect:unit_Y}, we construct the unit $\one_\Ydown$ of the desired symmetric monoidal structure $\oY$. However, we cannot yet prove that $\oY$ carries the associativity data required from a symmetric monoidal structure.

Much of this section is dedicated to showing that there is a unique way to extend the unitality and symmetry data of Sections~\ref{subsect:def_Y} and~\ref{subsect:unit_Y} into a symmetric monoidal structure on $\Mod(G)$. This is done in several steps, by showing that there is a unique choice of associativity data on larger and larger subcategories. In Section~\ref{subsect:Y_for_generic_reps}, we prove that $\one_\Ydown$ and $\oY$ can be taken to be a part of a symmetric monoidal structure on the full subcategory of $\Mod(G)$ generated under colimits by $\one_\Ydown$. This full subcategory is equivalent to the quotient of $\Mod(G)$ by the degenerate representations, and consists of what we refer to as the \emph{degeneracy-free} representations (which are closely related to generic representations). This establishes much of the needed data for the associativity of $\oY$.

In Subsection~\ref{subsect:Y_for_sph_reps}, we will extend this result further and establish the associativity data on all \emph{spherical} representations. Finally, in Subsection~\ref{subsect:assoc_Y_from_non_sense}, we prove the existence of a unique extension of $\oY$ into a symmetric monoidal structure on all of $\Mod(G)$. This will make use of the fact that the degeneracy-free representations and the spherical representations, taken together, generate all of $\Mod(G)$. We also discuss some nice consequences of our construction in Subsection~\ref{subsect:assoc_Y_from_non_sense}.

\subsection{Fixing Left and Right Actions} \label{subsect:def_Y}

Let us begin by properly defining the functor $\oY$.

The first thing we need to do is to carefully distinguish between left and right $G$-modules. The reason for this is that although it is possible to convert left $G$-modules and right $G$-modules into each other, there are a-priori two ways of doing so. One can turn a left $G$-module $(\pi_L,M)$ into a right $G$-module $(\pi_R,M)$ via either
\begin{equation} \label{eq:left_right_module_inverse}
    \pi_R(g)=\pi_L(g^{-1})\in\End(M)
\end{equation}
or
\begin{equation} \label{eq:left_right_module_transpose}
    \pi_R(g)=\pi_L(g^{T})\in\End(M).
\end{equation}
The composition of these two functors is the outer automorphism $g\mapsto g^{-T}$ of $G$, given by the inverse of the transposition map.

Either choice is good. The most commonly used choice in the literature appears to be \eqref{eq:left_right_module_inverse}. However, we would like to unambiguously talk about the action of the center of $G$ on a module in a way which is independent of whether we think of it as a left or a right module. This forces us to use \eqref{eq:left_right_module_transpose} for all our needs:
\begin{definition} \label{def:directions_for_G_modules}
    We let $\Mod(G)=\Mod^L(G)$ denote the category of smooth left $G$-modules over $\CC$. We similarly denote by $\Mod^R(G)$ the category of smooth right $G$-modules. We fix an equivalence of categories:
    \[
        I_{RL}\co\Mod^L(G)\ra\Mod^R(G)
    \]
    given by sending a left-module $(\pi_L,M)$ to a right module $(\pi_R,M)$ as in Equation~\eqref{eq:left_right_module_transpose}. We denote the inverse by
    \[
        I_{LR}\co\Mod^R(G)\ra\Mod^L(G).
    \]
\end{definition}

Because this is an unusual choice, we will be especially pedantic in explicitly applying the functors $I_{RL}$ and $I_{LR}$ in this section. Nevertheless, we will allow ourselves to be lax again once we get to Section~\ref{sect:T_global}.

Recall that the $G\times G\times G$-module $Y$ from Construction~\ref{const:middle_action}. Using this, we now construct:
\begin{construction}
    We have a functor
    \[
        \oY\co\Mod(G)\times\Mod(G)\ra\Mod(G)
    \]
    given by
    \[
        V\oY V'=\left(I_{RL}(V)\otimes I_{RL}(V')\right)\otimes_{G\times G}Y,
    \]
    with the $G\times G$ action on $Y$ being given by the left and right translation actions respectively, and where the left $G$-action on $V\oY V'$ is given by the middle action of $G$ on $Y$.
    
    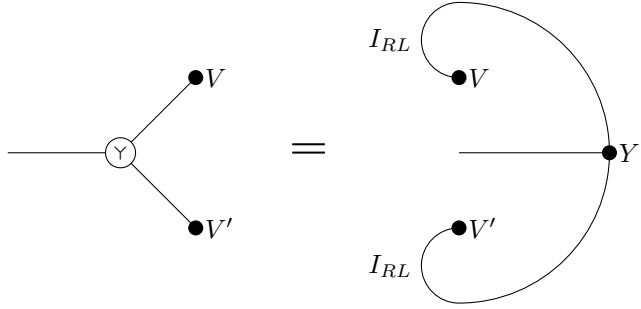
\begin{figure}
        \centering
        \begin{tikzpicture}
            \path[fill] (0,0) circle[radius=0.1];
            \node[right] at (0,0) {$Y$};
            \draw (0,0) arc[radius = 2, start angle=0, end angle=90];
            \draw (-2,2) arc[radius = 0.5, start angle=90, end angle=270];
            \node[left] at (-2.5,1.5) {$I_{RL}$};
            \path[fill] (-2,1) circle[radius=0.1];
            \node[right] at (-2,1) {$V$};
            
            \draw (0,0) -- (-2,0);
            
            \draw (0,0) arc[radius = 2, start angle=0, end angle=-90];
            \draw (-2,-1) arc[radius = 0.5, start angle=90, end angle=270];
            \node[left] at (-2.5,-1.5) {$I_{RL}$};
            \path[fill] (-2,-1) circle[radius=0.1];
            \node[right] at (-2,-1) {$V'$};
            
            \node at (-4,0) {{\huge $=$}};
            
            \path[fill] (-5.5,1) circle[radius=0.1];
            \node[right] at (-5.5,1) {$V$};
            \draw (-5.5,1) -- (-6.5+0.14,0+0.14);
            \path[fill] (-5.5,-1) circle[radius=0.1];
            \node[right] at (-5.5,-1) {$V'$};
            \draw (-5.5,-1) -- (-6.5+0.14,0-0.14);
            \draw (-6.5,0) circle[radius=0.2];
            \node at (-6.5,0) {$\Ydown$};
            \draw (-6.5-0.2,0) -- (-8,0);
        \end{tikzpicture}
        \caption{String diagram for the bi-functor $\oY$.}
        \label{fig:oY_string_diagram}
    \end{figure}
\end{construction}

\begin{remark} \label{remark:Y_is_commutative}
    Observe that the action of the transposition of $\M_2(F)$ on $Y=S(\M_2(F)\times F^\times)$ induces a natural symmetry isomorphism:
    \[
        V\oY V'\xrightarrow{\sim}V'\oY V.
    \]
    That is, the functor $\oY$ has a canonical commutativity constraint.
\end{remark}

\begin{remark} \label{remark:Y_is_comm_frob}
    In fact, when constructing $\oY$, we have seemingly made an arbitrary choice of which two of the three $G$-actions on $Y$ we use to match with a module, and a third through which we act on the result. By Proposition~\ref{prop:Y_three_actions_coincide}, this choice does not matter: the action of $G^3=G\times G\times G$ on $Y$ can be extended to an action of $S_3\ltimes G^3$, in a way compatible with Remark~\ref{remark:Y_is_commutative}.
\end{remark}

\begin{remark}
    The functor $\oY$ is colimit preserving. That is, it is right-exact and respects direct sums.
\end{remark}

\subsection{Unitality of \texorpdfstring{$\oY$}{Y}} \label{subsect:unit_Y}

In this subsection, we will derive some of the basic properties of the functor $\oY$. Specifically, we will introduce the unit $\one_\Ydown$ of the symmetric monoidal structure-to-be. Since we will need a name for this kind of structure (a symmetric product with a unit but with no associativity data), we begin by formally defining and naming it in subsubsection~\ref{subsubsect:weak_symm_mon}. We will construct this structure for the category $\Mod(G)$ in subsubsection~\ref{subsubsect:weak_symm_mon_mod_G}

\subsubsection{Weak Symmetric Monoidal Structures} \label{subsubsect:weak_symm_mon}

The definition we need is the following:
\begin{definition}
    Let $\C$ be a category. A \emph{weak symmetric monoidal structure} on $\C$ is a bi-functor
    \[
        \otimes\co\C\times\C\ra\C,
    \]
    an object
    \[
        \one_\C\in\C
    \]
    and natural isomorphisms
    \[
        s_{A,B}\co A\otimes B\xrightarrow{\sim} B\otimes A,\qquad l_A\co\one_\C\otimes A\xrightarrow{\sim}A,
    \]
    such that the following diagrams commute:
    \[\xymatrix{
        A\otimes B \ar[rd]^{s_{A,B}} \ar@{=}[rr]^{\id_{A\otimes B}} & & A\otimes B \\
        & B\otimes A \ar[ru]^{s_{B,A}} &
    }\]
    and
    \[\xymatrix{
        \one_\C\otimes\one_\C \ar[rr]^{s_{\one_\C,\one_\C}} \ar[rd]^{l_{\one_\C}} & & \one_\C\otimes\one_\C \ar[ld]_{l_{\one_\C}} \\
        & \one_\C. &
    }\]
    A category with a weak symmetric monoidal structure is a \emph{weak symmetric monoidal category}.
\end{definition}
\begin{definition}
    If $\C$ and $\D$ are two weak symmetric monoidal categories, then a \emph{weak symmetric monoidal functor} between them is a functor $F\co\C\ra\D$, along with natural isomorphisms
    \[
        \one_\D\xrightarrow{\sim}F(\one_\C),\qquad F(A)\otimes F(B)\xrightarrow{\sim}F(A\otimes B),
    \]
    such that the following diagrams commute:
    \[\xymatrix{
        F(A)\otimes F(B) \ar[d]^\sim \ar[rr]^{s_{F(A),F(B)}} & & F(B)\otimes F(A) \ar[d]^\sim \\
        F(A\otimes B) \ar[rr]^{F(s_{A,B})} & & F(B\otimes A)
    }\]
    and
    \[\xymatrix{
        \one_\D\otimes F(A) \ar[d]^\sim \ar[r]^-{l_{F(A)}} & F(A) \\
        F(\one_\C)\otimes F(A) \ar[r]^-\sim & F(\one_\C\otimes A). \ar[u]^{F(l_A)}
    }\]
\end{definition}

\begin{remark}
    One can use the language of operads to encode this data more cleanly. Indeed, we define the weak symmetric operad $\O_w$ to have $n$-ary operations indexed by binary trees with labeled leaves and unordered children. We allow the empty tree to be a $0$-ary operation. The composition of these operations is the natural composition of trees. For example, $\O_w$ admits a single commutative binary operation $a\cdot b$, and the three ternary operations in $\O_w$ are given by: $(a\cdot b)\cdot c$, $(b\cdot c)\cdot a$ and $(c\cdot a)\cdot b$.
    
    Note that the category of categories admits a Cartesian monoidal structure. With this monoidal structure, a weak symmetric monoidal category is an $\O_w$-algebra in categories. A weak symmetric monoidal functor is a morphism of $\O_w$-algebras in categories.
\end{remark}

\subsubsection{A Weak Symmetric Monoidal Structure on \texorpdfstring{$\Mod(\GL_2(F))$}{Mod(GL2(F))}} \label{subsubsect:weak_symm_mon_mod_G}

The structure of the rest of this subsection is as follows. We begin by constructing the unit $\one_\Ydown$ for $\oY$. Afterwards, we will show that $\one_\Ydown$ is indeed a unit, and satisfies the required axioms. This will allow us to turn $\oY$ into a weak symmetric monoidal structure on $\Mod(G)$.

Let $e\co F\ra\CC^\times$ be the same additive character used in Section~\ref{sect:GJ_vs_JL}, and recall the corresponding non-trivial additive character
\begin{equation*}
    \theta\co U\ra\CC^\times
\end{equation*}
on the subgroup $U=U_2(F)\subseteq\GL_2(F)$ of upper triangular unipotent matrices.
\begin{construction} \label{const:Y_unit}
    Let $\one_\Ydown$ be the space of $\theta$-co-invariants of $S(G)$ with respect to the right action. That is, we apply the functor $\Phi^-$ of \cite{derivatives_of_p_adic_repsI} to the right action of $G$ on the space $S(G)$.
\end{construction}

\begin{remark}
    See also Recollection~\ref{recollection:kirillov} for more context on this space.
\end{remark}

\begin{remark}
    Note that as it is defined in \cite{derivatives_of_p_adic_repsI}, the functor $\Phi^-$ carries modules of the mirabolic group $P_2\subseteq\GL_2(F)$ to modules of $P_1\subseteq\GL_1(F)$. Since the group $P_1$ is trivial, we may think of $\Phi^-$ as a functor $\Phi^-\co\Mod(G)\ra\Vect$.
\end{remark}

\begin{warning} \label{warn:left_right_translation}
    There are contradictory conventions with regard to what constitutes the ``right'' or ``left'' action on the space $S(G)$. With one convention, one has the right action acting by left-translation on the underlying space $G$. With another convention, one has an action by right-translation, but after a twist such as $g\mapsto g^{-1}$ or $g\mapsto g^T$. Unless explicitly stated otherwise, in this section we let $G$ act on $S(G)$ according to the first convention:
    \[
        (g_L\cdot f\cdot g_R)(h)=f(g_R hg_L).
    \]
\end{warning}

\begin{remark}
    To clarify, the space $\one_\Ydown$ consists of smooth compactly supported functions $f\co G\ra\CC$ modulo the relations of the form
    \[
        f(ug)-\theta(u)f(g)\sim 0
    \]
    for all $u\in U$. The space $\one_\Ydown$ carries a left $G$-action given by
    \[
        (\pi(h)\cdot f)(g)=f(gh).
    \]
\end{remark}

\begin{remark} \label{remark:one_Y_is_comp_ind}
    Let us relate $\one_\Ydown$ to (perhaps) better known spaces.  Fixing a Haar measure on $U$ allows us to define a canonical map:
    \[
        \one_\Ydown\ra\cInd_U^G(\theta),
    \]
    where $\cInd_U^G(\theta)$ is the representation of $G$ obtained from $\theta$ by induction with compact support. One can easily show that this map is bijective, and thus an isomorphism.
\end{remark}

The main result of this subsection is the unitality result:
\begin{construction} \label{const:one_Y_is_unit}
    There is a canonical natural isomorphism of functors from $\Mod(G)$ to $\Mod(G)$:
    \[
        \one_\Ydown\oY M\xrightarrow{\sim}M.
    \]
    
    Indeed, we use Remark~\ref{remark:Y_is_comm_frob} to identify
    \begin{align*}
        \one_\Ydown\oY M & =\left(I_{RL}(\one_\Ydown)\otimes I_{RL}(M)\right)\otimes_{G\times G}Y\cong \\
        & \cong I_{RL}(\one_\Ydown)\otimes_G\bigg(\left(S(G)\otimes I_{RL}(M)\right)\otimes_{G\times G}Y\bigg),
    \end{align*}
    where the tensor products $-\otimes_{G\times G}Y$ are with respect to the left and right translation actions, and the result is considered as a $G$-module via the middle action.
    
    However, we can construct an explicit isomorphism:
    \[
        \phi\co I_{RL}(\one_\Ydown)\otimes_G Y\xrightarrow{\sim} S(G),
    \]
    where the tensor product is taken with respect to the middle action. Indeed, there is an identification of functors $I_{RL}(\one_\Ydown)\otimes_G(-)\cong\Phi^-(-)$, and by the same argument as in Lemma~\ref{lemma:mirabolic_rest}, applying the functor $\Phi^{-}$ to the middle action on $Y$ gives the space $S(G)$, considered as a $G\times G$ left-module by applying $I_{LR}$ to its right action. This gives the required isomorphism, explicitly given by:
    \begin{equation} \label{eq:explicit_unitality}
        \phi\co f(h)\otimes\Psi(g,y)\mapsto\int_G f(h)\cdot\left(wh^{-T}\cdot\Psi\right)\left(g^{-1},-\det(g)\right)\dtimes{h},
    \end{equation}
    where $wh^{-T}\cdot\Psi$ denotes the middle action, and $w=\begin{pmatrix}& -1 \\ 1 &\end{pmatrix}$.
    
    We now get the desired isomorphism as a composition:
    \[
        \one_\Ydown\oY M\xrightarrow{\ref{remark:Y_is_comm_frob}} I_{RL}(\one_\Ydown)\otimes_G\bigg(\left(S(G)\otimes I_{RL}(M)\right)\otimes_{G\times G}Y\bigg)\xrightarrow{\phi} M.
    \]
\end{construction}

\begin{remark}
    Note that we have repeatedly used the property:
    \[
    S(G)\otimes_G N\cong N.
    \]
\end{remark}

\begin{figure}
    \centering
    \begin{tikzpicture}
        \path[fill] (-1.5,0) circle[radius=0.1];
        \node[right] at (-1.5,0) {$M$};
        \draw (-1.5,0) -- (-3,0);
        
        \node at (-4,0) {{\huge $\cong$}};
        
        \path[fill] (-5.5,1) circle[radius=0.1];
        \node[right] at (-5.5,1) {$M$};
        \draw (-5.5,1) -- (-6.5+0.14,0+0.14);
        \path[fill] (-5.5,-1) circle[radius=0.1];
        \node[right] at (-5.5,-1) {$\one_\Ydown$};
        \draw (-5.5,-1) -- (-6.5+0.14,0-0.14);
        \draw (-6.5,0) circle[radius=0.2];
        \node at (-6.5,0) {$\Ydown$};
        \draw (-6.5-0.2,0) -- (-8,0);
    \end{tikzpicture}
    \caption{String diagram illustrating Construction~\ref{const:one_Y_is_unit}.}
    \label{fig:one_Y_string_diagram}
\end{figure}
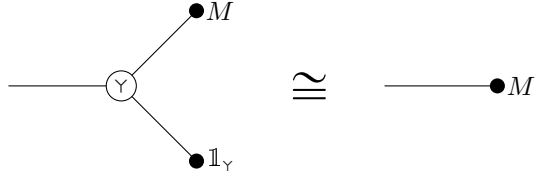

\begin{remark} \label{remark:phi_minus_is_trace}
    While performing Construction~\ref{const:one_Y_is_unit} above, we have in fact also constructed a canonical isomorphism:
    \[
        \Phi^-(M\oY N)\cong I_{RL}(M)\otimes_G N.
    \]
\end{remark}

\begin{figure}
    \centering
    \begin{tikzpicture}
        \path[fill] (-1.5,1) circle[radius=0.1];
        \node[right] at (-1.5,1) {$M$};
        \path[fill] (-1.5,-1) circle[radius=0.1];
        \node[right] at (-1.5,-1) {$N$};
        \draw (-1.5,1) arc[radius = 1, start angle=90, end angle=270];
        \node[left] at (-2.5,0) {$I_{RL}$};
        
        \node at (-4,0) {{\huge $\cong$}};
        
        \path[fill] (-5.5,1) circle[radius=0.1];
        \node[right] at (-5.5,1) {$M$};
        \draw (-5.5,1) -- (-6.5+0.14,0+0.14);
        \path[fill] (-5.5,-1) circle[radius=0.1];
        \node[right] at (-5.5,-1) {$N$};
        \draw (-5.5,-1) -- (-6.5+0.14,0-0.14);
        \draw (-6.5,0) circle[radius=0.2];
        \node at (-6.5,0) {$\Ydown$};
        \draw (-6.5-0.2,0) -- (-8,0);
        \path[fill] (-8,0) circle[radius=0.1];
        \node[left] at (-8,0) {$\Phi^-$};
    \end{tikzpicture}
    \caption{String diagram illustrating Remark~\ref{remark:phi_minus_is_trace}.}
    \label{fig:phi_minus_string_diagram}
\end{figure}
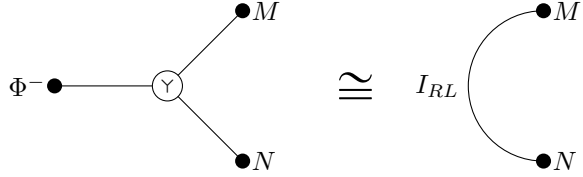

The natural isomorphism of Construction~\ref{const:one_Y_is_unit}, together with the commutativity structure of Remark~\ref{remark:Y_is_commutative}, give a composition
\begin{equation} \label{eq:braiding_on_one_Y} \xymatrix{
    \one_\Ydown \ar@{-->}@/^2pc/[rrr]^\lambda & \one_\Ydown\oY\one_\Ydown \ar[l]_{\ref{const:one_Y_is_unit}} \ar[r]^{\ref{remark:Y_is_commutative}} & \one_\Ydown\oY\one_\Ydown \ar[r]^{\ref{const:one_Y_is_unit}} & \one_\Ydown.
}\end{equation}
The automorphism $\lambda\co\one_\Ydown\ra\one_\Ydown$ is self-inverse by construction. In fact, it turns out that:
\begin{claim} \label{claim:oY_left_right_unit}
    The automorphism $\lambda\co\one_\Ydown\xrightarrow{\sim}\one_\Ydown$ of Diagram~\eqref{eq:braiding_on_one_Y} is the identity map.
\end{claim}

\begin{remark}
    Claim~\ref{claim:oY_left_right_unit} would follow at once if we already knew that $\oY$ and $\one_\Ydown$ were a part of a symmetric monoidal structure. However, we need to prove this directly.
\end{remark}

\begin{proof}[Proof of Claim~\ref{claim:oY_left_right_unit}]
    We are essentially asking about the triviality of the braiding action of $S_2$ on $\one_\Ydown\oY\one_\Ydown$. Therefore, we are asking about the action of the transposition of $\M_2(F)$ on
    \[
        (I_{RL}(\one_\Ydown)\otimes I_{RL}(\one_\Ydown))\otimes_{G\times G}Y.
    \]
    
    That is, we want to show that the $\theta$-co-invariants of $S(\M_2(F)\times F^\times)$ under the left and right actions of the group $U$ are invariant under the transposition of $\M_2(F)$. We will show this using a standard technique by Gelfand--Kazhdan (see for example Theorem~6.10 of \cite{padic_reps_survey}). Since the author lacks a specific reference applicable to the current use-case, and because this technique is usually used for $\GL_2(F)$ rather than $\M_2(F)$, let us write the details of its application here explicitly. Specifically, we will decompose $\M_2(F)\times F^\times$ into $U\times U$-orbits, and prove that the only ones which carry a $\theta$-equivariant distribution are also invariant under the transposition.
    
    Indeed, we can decompose $Y$ as a $G\times G$-module into a filtration with graded pieces
    \[
        S(G\times F^\times),\qquad S(\M_2^{\det=0}(F)\times F^\times).
    \]
    We will show that the transposition acts by identity on the tensor product of both pieces with $\one_\Ydown$ from both sides (note that the functor $(-)\otimes_G\one_\Ydown$ is exact, as the functor $\Phi^-$ is exact).
    
    We begin with the space $S(G\times F^\times)$, which can itself be decomposed as a $U\times U$-module into
    \[
        S\left(U\begin{pmatrix} a & \\ & b \end{pmatrix}U^T\times F^\times\right),\qquad S\left(U\begin{pmatrix} & a \\ b & \end{pmatrix}U^T\times F^\times\right).
    \]
    Transposition acts as the identity on the first of these two components, and the other is supported on $\{a=b\}$ after multiplication by $\one_\Ydown$ on both sides.
    
    This leaves us with the other component, $S(\M_2^{\det=0}(F)\times F^\times)$. It
    is clear, by decomposing $\M_2^{\det=0}(F)$ into $U\times U$-orbits as above, that $S(\M_2^{\det=0}(F)\times F^\times)$ becomes isomorphic to $S(F^\times\times F^\times)$ after multiplication by $\one_\Ydown$ on both sides. The isomorphism is given explicitly by
    \[
        \Psi(g,y)\mapsto \int\int\Psi\left(\begin{pmatrix}1 & u \\ & 1\end{pmatrix}\begin{pmatrix}0 & \\ & z\end{pmatrix}\begin{pmatrix}1 & \\ v & 1\end{pmatrix},\,y\right)e(u+v)\d{u}\d{v},
    \]
    and it is clearly symmetric to the transposition.
\end{proof}

\begin{figure}
    \centering
    \begin{tikzpicture}
        \draw (-1.5+0.14,0+0.14) to [out=45,in=180] (-0.5,1) to [out=0,in=180] (1.5,-1);
        \path[fill] (1.5,-1) circle[radius=0.1];
        \node[right] at (1.5,-1) {$\one_\Ydown$};
        \draw (-1.5+0.14,0-0.14) to [out=-45,in=180] (-0.5,-1) to [out=0,in=180] (1.5,1);
        \path[fill] (1.5,1) circle[radius=0.1];
        \node[right] at (1.5,1) {$\one_\Ydown$};
        \draw (-1.5,0) circle[radius=0.2];
        \node at (-1.5,0) {$\Ydown$};
        \draw (-1.5-0.2,0) -- (-3,0);
        
        \node at (-4,0) {{\huge $\cong$}};
        
        \path[fill] (-5.5,1) circle[radius=0.1];
        \node[right] at (-5.5,1) {$\one_\Ydown$};
        \draw (-5.5,1) -- (-6.5+0.14,0+0.14);
        \path[fill] (-5.5,-1) circle[radius=0.1];
        \node[right] at (-5.5,-1) {$\one_\Ydown$};
        \draw (-5.5,-1) -- (-6.5+0.14,0-0.14);
        \draw (-6.5,0) circle[radius=0.2];
        \node at (-6.5,0) {$\Ydown$};
        \draw (-6.5-0.2,0) -- (-8,0);
    \end{tikzpicture}
    \caption{String diagram illustrating Claim~\ref{claim:oY_left_right_unit}. The claim is that the two ways of identifying the two sides of the diagram (using the symmetry of $\oY$ compared with the symmetry of the usual tensor product $\otimes$) give the same isomorphism.}
    \label{fig:one_Y_commutativity_string_diagram}
\end{figure}
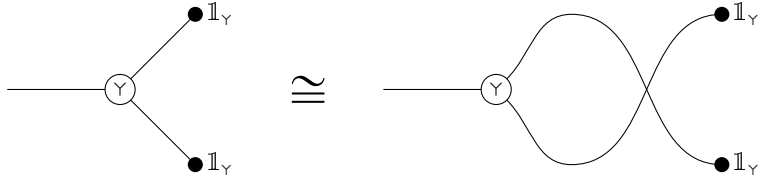

\begin{corollary} \label{cor:one_Y_is_commutative}
    The ring of endomorphisms
    \[
        \End(\one_\Ydown)
    \]
    is commutative.
\end{corollary}
\begin{proof}
    This is given by a minor variant of the classical Eckmann-Hilton argument. Indeed, we note that there are two compatible multiplications on $\End(\one_\Ydown)$: composition and the $\oY$ product. We do not yet know that the latter is associative, but the Eckmann-Hilton argument works regardless; all we actually need is the triviality of the braiding shown in Claim~\ref{claim:oY_left_right_unit}.
\end{proof}

\begin{corollary}
    The bi-functor $\oY\co\Mod(G)\times\Mod(G)\ra\Mod(G)$, together with the object $\one_\Ydown$, the isomorphism of Construction~\ref{const:one_Y_is_unit} and the commutativity structure of Remark~\ref{remark:Y_is_commutative}, form a weak symmetric monoidal structure on $\Mod(G)$.
\end{corollary}

\subsection{Restriction to Degeneracy-Free Representations} \label{subsect:Y_for_generic_reps}

In this subsection, we will study the behaviour of the product $\oY$ on \emph{degeneracy-free} representations. These will turn out to be the representations generated under colimits from the unit $\one_\Ydown$. Our main result (Theorem~\ref{thm:oY_gen_is_over_center}) will be that $\oY$ behaves as a tensor product over the center of $\Mod(G)$ when restricted to degeneracy-free representations, and in particular it can be extended uniquely to a symmetric monoidal structure on these representations.

The basic definition we will use is:
\begin{definition}
    We say that a module $M\in\Mod(G)$ is \emph{degenerate} if
    \[
        \Phi^-(M)=0.
    \]
\end{definition}

\begin{definition} \label{def:deg_free_reps}
    We say that a module $N\in\Mod(G)$ is \emph{degeneracy-free} if
    \[
        \Ext^i(N,M)=0
    \]
    for all degenerate $M\in\Mod(G)$ and $i=0,1$. We denote by $\Mod^{\dgfr}(G)\subseteq\Mod(G)$ the full subcategory of degeneracy-free representations.
\end{definition}

\begin{remark} \label{remark:desc_of_dgfr}
    Definition~\ref{def:deg_free_reps} deserves a more conceptual explanation. There are two equivalent descriptions of $\Mod^{\dgfr}(G)$ that will follow from the results later in this subsection:
    \begin{enumerate}
        \item The category $\Mod^{\dgfr}(G)$ is the smallest full subcategory of $\Mod(G)$ which contains $\one_\Ydown$ and is closed under colimits. This will be Claim~\ref{claim:one_Y_generic}.
        \item The second description is as follows. Consider the quotient $\C$ of the category $\Mod(G)$ by the degenerate representations. Then the quotient functor $\Mod(G)\ra\C$ has a left adjoint, and this left adjoint defines an equivalence between $\C$ and $\Mod^{\dgfr}(G)$. This will follow from Claim~\ref{claim:one_Y_is_phi_minus} and Theorem~\ref{thm:oY_gen_is_over_center}.
    \end{enumerate}
\end{remark}

\begin{warning}
    Note that while most generic irreducible representations are degeneracy-free, this is not always the case. Specifically, let $\St$ be the Steinberg representation with trivial central character. While $\St$ is generic, it is not degeneracy-free. Instead, let us denote by $\binom{\St}{\one_G}$ the unique non-trivial extension of the trivial representation $\St$ by $\one_G$. Then $\binom{\St}{\one_G}$ \emph{is} degeneracy-free.
\end{warning}

For each Bernstein component $c$ of $\Mod(G)$, let $Z^c$ be the corresponding Bernstein center. Let
\[
    \Z=\bigoplus_c Z^c
\]
be the direct sum of all $Z^c$-s. This is a non-unital ring, but it is \emph{quasi-unital} (this follows formally because $\Z$ is a direct sum of unital rings). This means that it satisfies that the multiplication map
\[
    \Z\otimes_\Z\Z\ra\Z
\]
is an isomorphism. Here, $\otimes_\Z$ denotes the relative tensor product over $\Z$: that is, for a pair of (non-unital) $\Z$-modules $M$ and $N$, the space $M\otimes_\Z N$ is the quotient of $M\otimes N$ by the expressions of the form
\[
    zm\otimes n-m\otimes zn,\qquad\qquad m\in M,\, n\in N,\, z\in\Z.
\]

Denote the symmetric monoidal category of \emph{smooth} (non-unital) $\Z$-modules by $\Mod(\Z)$, equipped with the relative tensor product over $\Z$. By definition, a $\Z$-module $M$ is smooth if it satisfies that the action map
\[
    \Z\otimes_\Z M\ra M
\]
is an isomorphism. Essentially, this condition means that each $m\in M$ is supported on finitely many components of $\Z$. Intuitively, quasi-unital rings share many properties with unital rings, and their smooth modules take the role of unital modules. See also Section~3 of \cite{dixmier_malliavin_for_born_arxiv} for more details about quasi-unital rings and their smooth modules.

Note that the symmetric monoidal category $\Mod(\Z)$ acts on $\Mod(G)$. This is a fancy way of saying that the center $Z^{c}$ of each component $c$ automatically acts on the objects of the component $\Mod^{c}(G)$. Our main theorem for this subsection is:
\begin{theorem} \label{thm:oY_gen_is_over_center}
    The functor
    \[
        -\otimes_\Z\one_\Ydown\co\Mod(\Z)\ra\Mod(G)
    \]
    is fully faithful, with essential image $\Mod^{\dgfr}(G)$. Moreover, this functor is weak symmetric monoidal.
\end{theorem}

\begin{corollary} \label{cor:generic_oY_assoc}
    Consider the restriction
    \begin{equation*}
        \oY\co\Mod^{\dgfr}(G) \times\Mod^{\dgfr}(G)\ra\Mod^{\dgfr}(G)
    \end{equation*}
    of the weak symmetric monoidal structure of $\Mod(G)$ to $\Mod^\dgfr(G)$. Then it can be extended from a weak symmetric monoidal structure on $\Mod^{\dgfr}(G)$ into a symmetric monoidal structure on $\Mod^{\dgfr}(G)$ in an essentially unique way (that is, in a way that is unique up to a unique isomorphism).
\end{corollary}

\begin{remark}
    Recall from Remark~\ref{remark:desc_of_dgfr} that we can think of $\Mod^\dgfr(G)$ both as a sub-category of $\Mod(G)$ and as a quotient (in fact, a co-localization) of it. Theorem~\ref{thm:oY_gen_is_over_center} says that as a \emph{sub-category}, the category $\Mod^\dgfr(G)$ is closed under $\oY$. However, the reader should note that the multiplication $\oY$ does \emph{not} descend to $\Mod^\dgfr(G)$ as a \emph{quotient}. Specifically, the trivial representation $\one_G$ is killed by this quotient (as it is degenerate), but one can show that
    \[
        \one_G\oY\one_G=\binom{\one_G}{\St}
    \]
    is an extension of $\one_G$ by a Steinberg representation (so that $\one_G$ is the quotient). However, $\binom{\one_G}{\St}$ is non-degenerate, as $\Phi^-$ is an exact functor (see Proposition~3.2(a) of \cite{derivatives_of_p_adic_repsI}) and $\St$ is non-degenerate. Thus, there are products of degenerate representations that are not degenerate.
\end{remark}

\subsubsection{Proof of Theorem~\ref{thm:oY_gen_is_over_center}}

The rest of this subsection is structured as follows. After introducing some notation, we will state Claims~\ref{claim:one_Y_is_phi_minus}, \ref{claim:end_of_one_Y_is_cent} and \ref{claim:one_Y_generic}. This will let us prove Theorem~\ref{thm:oY_gen_is_over_center}. Finally, we will prove the three claims.

Observe that the functor
\[
    \Phi^-\co\Mod(G)\ra\Vect
\]
factors through the forgetful functor $\Mod(\Z)\ra\Vect$:
\[\xymatrix{
    \Mod(G) \ar[r]^-{\Phi^-} \ar@{-->}[d]_{\Phi_{\mathrm{enh}}^-} & \Vect \\
    \Mod(\Z). \ar[ur]
}\]
In other words, because $\Z$ acts on $\Mod(G)$, the spaces $\Phi^-(V)$ for $V\in\Mod(G)$ are canonically $\Z$-modules. We denote the corresponding $\Z$-module by $\Phi_{\mathrm{enh}}^-(V)$, and say that $\Phi_{\mathrm{enh}}^-$ is an \emph{enhancement} of $\Phi^-$.

\begin{remark} \label{remark:one_Y_is_proj}
    The functor $\Phi^-$ respects all colimits, as it has a right adjoint (by Proposition~3.2 of \cite{derivatives_of_p_adic_repsI}). In particular, $\Phi_{\mathrm{enh}}^-\co\Mod(G)\ra\Mod(\Z)$ also respects all colimits.
\end{remark}

For $M,N\in\Mod(G)$, let
\[
    \sHom(M,N)\subseteq\Hom(M,N)
\]
be the subset consisting of morphisms supported on finitely many Bernstein components. This is the enriched inner $\Hom$ of $\Mod(G)$ over $\Mod(\Z)$, and we therefore think of it as an object of $\Mod(\Z)$. Note that $\sHom(M,N)$ is also given by the formula:
\[
    \sHom(M,N)=\Z\otimes_{Z(\Mod(G))}\Hom(M,N)
\]
where $Z(\Mod(G))$ is the usual Bernstein center of $\Mod(G)$.

We will prove Theorem~\ref{thm:oY_gen_is_over_center} by the combination of the following three claims:

\begin{claim} \label{claim:one_Y_is_phi_minus}
    The functor
    \[
        -\otimes_\Z\one_\Ydown\co\Mod(\Z)\ra\Mod(G)
    \]
    is left adjoint to
    \[
        \Phi_{\mathrm{enh}}^-\co\Mod(G)\ra\Mod(\Z).
    \]
\end{claim}

\begin{claim} \label{claim:end_of_one_Y_is_cent}
    The natural map of rings induced by the action of $\Z$ on $\Mod(G)$:
    \[
        \Z\ra\sHom(\one_\Ydown,\one_\Ydown)
    \]
    is an isomorphism.
\end{claim}

\begin{claim} \label{claim:one_Y_generic}
    The smallest full subcategory of $\Mod(G)$ containing $\one_\Ydown$ and closed under all colimits is $\Mod^{\dgfr}(G)$.
\end{claim}

\begin{remark} \label{remark:one_Y_is_phi_minus_explicit}
    Note that Claim~\ref{claim:one_Y_is_phi_minus} is equivalent to the statement that there is a natural isomorphism of functors:
    \[
        \Phi_{\mathrm{enh}}^-M\xrightarrow{\sim}\sHom(\one_\Ydown,M).
    \]
\end{remark}

Using Claims~\ref{claim:one_Y_is_phi_minus}, \ref{claim:end_of_one_Y_is_cent} and \ref{claim:one_Y_generic}, we can now prove our main theorem.
\begin{proof}[Proof of Theorem~\ref{thm:oY_gen_is_over_center}]
    Claim~\ref{claim:one_Y_is_phi_minus} implies that the functor:
    \[
        -\otimes_\Z\one_\Ydown\co\Mod(\Z)\ra\Mod(G)
    \]
    is left adjoint to $\Phi_{\mathrm{enh}}^-$. Remark~\ref{remark:one_Y_is_proj} and Claim~\ref{claim:end_of_one_Y_is_cent} show that the unit
    \[
        N\ra\Phi_{\mathrm{enh}}^-(N\otimes_\Z\one_\Ydown)
    \]
    is an isomorphism, because $\Z\cong\Phi_{\mathrm{enh}}^-\one_\Ydown$ by Claim~\ref{claim:end_of_one_Y_is_cent} and Remark~\ref{remark:one_Y_is_phi_minus_explicit}.
    
    The characterization of the essential image follows from Claim~\ref{claim:one_Y_generic}.
    
    The isomorphism of Construction~\ref{const:one_Y_is_unit} now induces the requisite natural isomorphism between $\otimes_\Z$ and $\oY$.
\end{proof}

Let us now prove the three claims.

\begin{proof}[Proof of Claim~\ref{claim:one_Y_is_phi_minus}]
    We want to construct a natural isomorphism:
    \[
        \Phi_{\mathrm{enh}}^-M\xrightarrow{\sim}\sHom(\one_\Ydown,M),
    \]
    for $M\in\Mod(G)$. Using the Haar measure on $U$ (as in Remark~\ref{remark:one_Y_is_comp_ind}), we can construct the map
    \begin{align*}
        \one_\Ydown\otimes I_{RL}(\one_\Ydown) & \ra S(G) \\
        f\otimes f' & \mapsto \int_U\int_G f(ug'g)f'(wg'^{-T})\theta(u^{-1})\dtimes{g'}\d{u},
    \end{align*}
    whose image lies in $S(G)$ by the smoothness of $f$ and $f'$. This map factors through the relative tensor product over the center of $\Mod(G)$, and thus induces a map of modules over the center:
    \[
        \Phi_{\mathrm{enh}}^-M\ra\Hom(\one_\Ydown,M).
    \]
    This map factors through $\sHom(\one_\Ydown,M)$.
    
    Therefore, we have a map
    \begin{equation} \label{eq:map_phi_minus_to_hom_space}
        \Phi_{\mathrm{enh}}^-M\ra\sHom(\one_\Ydown,M),
    \end{equation}
    which we want to prove is an isomorphism. Note that it is sufficient to prove this for $M$ supported on a single component $c$, in which case $\sHom(\one_\Ydown,M)\cong\Hom(\one_\Ydown,M)$.
    
    The map of Equation~\eqref{eq:map_phi_minus_to_hom_space} is very nearly an isomorphism for reasons of abstract non-sense. That is, a na\"ive expectation would be that this immediately follows by Frobenius reciprocity. The problem, essentially, is about the smoothness of the module $M$. That is, the object $\Hom(\one_\Ydown,M)$ is given by the $\theta$-invariants of the \emph{roughening} of $M$, defined to be the non-smooth $S(G)$-module $\Hom(S(G),M)$.
    
    However, by the description of the left adjoint to $\Phi^-$ as a functor $\Mod(P)\ra\Vect$, given in Subsection~3.2 of \cite{derivatives_of_p_adic_repsI}, we get
    \begin{multline*}
        \Phi^-\left(S(P)\otimes_{S(P)}\Hom_G(S(G),M)\right)\cong \\
        \cong \Hom_P\left(S(F^\times),S(P)\otimes_{S(P)}\Hom_G(S(G),M)\right)\cong \\
        \cong \Hom_P\left(S(P),S(P)\otimes_{S(P)}\Hom_G(S(G),M)\right)^{U,\theta}\cong \\
        \cong \Hom_G(S(G),M)^{U,\theta}\xrightarrow{\sim}\Hom_G(\one_\Ydown,M),
    \end{multline*}
    where $S(F^\times)$ is identified with the space of $\theta$-co-invariants of $S(P)$ with respect to the right action of $U$ on $S(P)$.
    
    Therefore, the claim follows from Lemma~\ref{lemma:partial_roughening} below.
\end{proof}

\begin{lemma} \label{lemma:partial_roughening}
    Let $M\in\Mod^c(G)$ belong to a single component $c$ of $\Mod(G)$. Then the natural map
    \begin{equation} \label{eq:G_smooth_vs_P_smooth}
        M\ra S(P)\otimes_{S(P)}\Hom(S(G),M)
    \end{equation}
    becomes an isomorphism after applying $\Phi^-$.
\end{lemma}
\begin{proof}
    We must prove that the cokernel of the map in Equation~\eqref{eq:G_smooth_vs_P_smooth} is $U$-invariant.
    
    To do this, it is sufficient to consider a representation $M$ given by compact induction:
    \[
        M=\cInd_K^G(\sigma)
    \]
    from a representation $(\sigma,V)$ of $K/K_0$, where $K=\GL_2(\O_F)$ and $K_0\subseteq K$ is some fixed normal compact open subset.
    
    In this case, $S(P)\otimes_{S(P)}\Hom(S(G),M)$ is the space of functions $f\co G\ra V$ which are $\sigma$-equivariant on the left, $P$-smooth on the right, and such that the convolution $\one_{K'}*f$ from the right is compactly supported for all compact open $K'\subseteq K$.
    
    Pick a sufficiently large compact open subset $U^{(0)}\subseteq U$. Replace $f$ with its average $\tilde{f}$ given by:
    \[
        \tilde{f}(g)=\int_{U^{(0)}}f(gu)\theta(u^{-1})\d{u}.
    \]
    Then $\tilde{f}(g)$ is supported on the $g$-s of the form:
    \[
        k\begin{pmatrix}a & \\ & d\end{pmatrix}\begin{pmatrix}1 & b \\ & 1\end{pmatrix},\qquad k\in K
    \]
    with bounded $\abs{b}$ and $\abs{d/a}$. However, this is sufficient to guarantee that $\tilde{f}$ is $G$-smooth on the right, and thus compactly supported. This means that $\tilde{f}$ is already in $M$.
\end{proof}

\begin{proof}[Proof of Claim~\ref{claim:end_of_one_Y_is_cent}]
    It is sufficient to prove that the map
    \[
        Z(\Mod(G))\ra\End(\one_\Ydown)
    \]
    from the center of $\Mod(G)$ to the endomorphisms of $\one_\Ydown$ is an isomorphism.
    
    Observe that the natural isomorphism
    \[
        \one_\Ydown\oY M\xrightarrow{\sim}M
    \]
    of Construction~\ref{const:one_Y_is_unit} induces a map from $\End(\one_\Ydown)$ to the endomorphisms of the identity functor of $\Mod(G)$, which gives a section to the above map.
    
    Therefore, it remains to prove that an element of the center of $\Mod(G)$ which maps to $0$ in $\End(\one_\Ydown)$ is $0$. Indeed, $\one_\Ydown$ maps surjectively onto generic irreducible representations, and elements of the center of $\Mod(G)$ can be tested for equality on each irreducible generic representation separately.
\end{proof}

\begin{proof}[Proof of Claim~\ref{claim:one_Y_generic}]
    It is clear that $\Mod^{\dgfr}(G)$ is closed under colimits and contains $\one_\Ydown$. Let $M\in\Mod^{\dgfr}(G)$. We want to show that it is generated under colimits from $\one_\Ydown$. Note that Claim~\ref{claim:one_Y_is_phi_minus} and Remark~\ref{remark:one_Y_is_proj} imply that $\one_\Ydown$ is projective. Thus, the cokernel of:
    \[
        p\co\one_\Ydown\otimes\Hom(\one_\Ydown,M)\ra M
    \]
    is degenerate. Since $\Hom(M,N)=0$ for all degenerate $N$, we conclude that $p$ is surjective.
    
    Let the kernel of $p$ be $K$. Since $K$ also satisfies that $\Hom(K,N)=0$ for all degenerate $N$, we see that $M$ can be written as the colimit:
    \[\xymatrix{
        \one_\Ydown\otimes\Hom(\one_\Ydown,K) \ar[r] & \one_\Ydown\otimes\Hom(\one_\Ydown,M) \ar[r] & M \ar[r] & 0.
    }\]
\end{proof}

\subsection{Restriction to Spherical Representations} \label{subsect:Y_for_sph_reps}

Our ultimate goal in this section is to turn $\oY$ into a symmetric monoidal structure on the whole of $\Mod(G)$ in a canonical way. Corollary~\ref{cor:generic_oY_assoc} already tells us that this is true on a large subcategory of $\Mod(G)$. However, it turns out that this is not enough to guarantee associativity in general. In fact, it is not difficult to construct a different weak symmetric monoidal structure $\oY'$ which has the same unit as $\oY$ and identifies with it on $\Mod^\dgfr(G)$, but which is not associative.

Thus, in order to prove associativity for $\oY$ in general, we must find a larger category on which we know that $\oY$ is associative. Once we have found such a sufficiently large category, the associativity of $\oY$ will become a property that can be checked, instead of additional data (see Subsection~\ref{subsect:assoc_Y_from_non_sense}).

In Subsection~\ref{subsect:Y_for_generic_reps}, we studied the restriction of $\oY$ to the subcategory generated by $\one_\Ydown$. It makes sense that this was not sufficient, as $\one_\Ydown$ does not generate all of $\Mod(G)$. Thus, our goal for this subsection is to study the restriction of $\oY$ to the subcategory generated by another projective object $E$, such that together with $\one_\Ydown$ they generate all of $\Mod(G)$.

Our choice of $E$ will eventually (Remark~\ref{remark:E_is_assoc}) turn out to be a co-algebra
\[
    \varepsilon\co E\ra\one_\Ydown,\qquad \mu\co E\ra E\oY E
\]
such that the co-multiplication map $\mu$ is an isomorphism. This will force $E$ to be co-associative, which will enable us to uniquely define the associativity data of $\oY$ in Subsection~\ref{subsect:assoc_Y_from_non_sense}.

\begin{definition}
    Denote $\O=\O_F$, and let $K$ be the maximal compact subgroup $\GL_2(\O)$. We let $K$ act on the space $S(\O^\times)$ by the determinant action. Let
    \[
        E=\cInd_K^G(S(\O^\times))
    \]
    be the $G$-module given by induction with compact support of $S(\O^\times)$.
    
    We denote the smallest full subcategory of $\Mod(G)$ which contains $E$ and is closed under colimits by $\Mod^\sph(G)$. We refer to objects of $\Mod^\sph(G)$ as \emph{spherical} representations.
\end{definition}

\begin{remark}
    Note that the term ``spherical representations'' is not standard terminology.
\end{remark}

\begin{remark}
    We can identify:
    \[
        E=\cInd_K^G(S(\O^\times))=\cInd_K^G(\cInd_{\SL_2(\O)}^K (\one_{\SL_2(\O)}))=\cInd_{\SL_2(\O)}^G\left(\one_{\SL_2(\O)}\right).
    \]
\end{remark}

\begin{remark} \label{remark:maps_from_E}
    We have the natural adjunction:
    \[
        \Hom(E,M)\cong \Hom(S(G),M)^{\SL_2(\O)}.
    \]
    As a result, in order to define a map from $E$ to a $G$-module $M$, one needs to give a $\SL_2(\O)$-invariant element of the \emph{roughening} of $M$, defined to be the (non-smooth) $S(G)$-module $\Hom(S(G),M)$. When $M$ is given by functions on some space, we will sometimes abuse notation and think of such elements as $\SL_2(\O)$-invariant distributions.
\end{remark}

\begin{remark}
    Conceptually, the reason that $E$ turns out to be a co-commutative co-algebra is related to the multiplicity one property of spherical vectors.
\end{remark}

\begin{remark}
    Observe that $E$ is supported only on Bernstein components of $\Mod(G)$ which contain a one-dimensional representation.
\end{remark}

Let
\[
    \Z_\sph=\bigoplus_{\text{$c$ spherical}} Z^c
\]
be the direct sum of the Bernstein centers of all components that contain a one-dimensional representation.  We think of $\Z_\sph$ as a smooth $\Z$-module, with $\Z$ as in Subsection~\ref{subsect:Y_for_generic_reps}. Note that $\Z_\sph$ is also quasi-unital.
Denote the symmetric monoidal category of smooth $\Z_\sph$-modules by $\Mod(\Z_\sph)$, and observe that $\Mod(\Z)$ acts on $\Mod(\Z_\sph)$.

Our main theorem for this subsection is:
\begin{theorem} \label{thm:oY_sph_is_over_center}
    The functor
    \[
        -\otimes_{\Z_\sph}E\co\Mod(\Z_\sph)\ra\Mod(G)
    \]
    is fully faithful, with essential image $\Mod^{\sph}(G)$. Moreover, this functor sends the unit $\Z_\sph\in\Mod(\Z_\sph)$ and relative tensor product $\otimes_{\Z_\sph}$ to $E$ and $\oY$ respectively.
\end{theorem}
\begin{remark}
    We do not actually claim that the functor of Theorem~\ref{thm:oY_sph_is_over_center} is weak symmetric monoidal, as it does not respect the unit.
\end{remark}

The proof of this theorem follows the same outline as the proof of Theorem~\ref{thm:oY_gen_is_over_center}, only easier. It is clear that $E$ is projective, that its restriction to each component is compact, and that
\[
    \Z_\sph\ra\sHom(E,E)
\]
is an isomorphism. Thus, Theorem~\ref{thm:oY_sph_is_over_center} would immediately follow if we could prove the existence of an isomorphism:
\[
    E\oY E\cong E.
\]
This isomorphism will be the content of Claim~\ref{claim:spherical_is_idemp}.

We begin by constructing a co-unit map $\varepsilon\co E\ra\one_\Ydown$, as well as constructing a co-multiplication map $\mu\co E\ra E\oY E$. The co-multiplication $\mu$ will then be the desired isomorphism $E\cong E\oY E$. Recall that $e\co F\ra\CC^\times$ is our chosen additive character.
\begin{construction} \label{const:E_co_unit}
    Let $\nu(e)$ be the largest integer such that $e(\pi^{-\nu(e)}\O)=1$, with $\pi\in\O$ a uniformizer. We let the map
    \[
        \varepsilon\co E\ra\cInd_U^G(\theta)\cong\one_\Ydown
    \]
    be given as follows.
    
    Let $f$ be the distribution on $G$ given by the delta distribution:
    \[
        f(g)=e(b)\delta_1(\det(k))
    \]
    if $g$ has the form
    \[
        \begin{pmatrix}1 & b \\ & 1\end{pmatrix}\begin{pmatrix}1 & \\ & -\pi^{\nu(e)}\end{pmatrix}k
    \]
    with $k\in\GL_2(\O)$, and we set $f(g)=0$ otherwise. It is clear that $f$ is compactly supported modulo $U$, left-$\SL_2(\O)$-invariant, and right-$\theta$-equivariant.
    
    Therefore, $f$ defines an $\SL_2(\O)$-invariant element of the roughening
    \[
        \Hom(S(G),\cInd_U^G(\theta)).
    \]
    By the universal property in Remark~\ref{remark:maps_from_E}, we conclude that the distribution $f$ defines a map $\varepsilon\co E\ra\cInd_U^G(\theta)$.
\end{construction}

\begin{construction} \label{const:E_co_mult}
    Let ${Y}^{\SL_2(\O)\times\SL_2(\O)}$ be the space of vectors in $Y$ which are invariant under the left and right actions of $\SL_2(\O)$. We let $\mu$ be the map
    \[
        \mu\co E\ra Y^{\SL_2(\O)\times\SL_2(\O)}\cong E\oY E
    \]
    given as in Remark~\ref{remark:maps_from_E} by the $S_3\ltimes \SL_2(\O)^3$-invariant distribution (more precisely, $\SL_2(\O)$-invariant element of the roughening of $Y^{\SL_2(\O)\times\SL_2(\O)}$):
    \[
        q^{-\nu(e)}\cdot \one_{\M_2(\O)}(g)\delta_{\pi^{-\nu(e)}}(y)
    \]
    for $(g,y)\in\M_2(F)\times F^\times$.
\end{construction}

\begin{remark} \label{remark:E_co_mult_is_dist_test_func}
    Let us try to informally motivate Constructions~\ref{const:E_co_unit} and~\ref{const:E_co_mult}. Because the co-multiplication determines the co-unit, we can focus on trying to understand the map $\mu$. It turns out that $\mu$ is related to the standard ``distinguished test function'' on $\M_2(F)$ appearing in the theory of Godement--Jacquet $L$-functions.
    
    Let us give more details. In the theory of Godement--Jacquet $L$-functions, it is known that if one wants to define the $L$-function of an unramified representation, then it is sufficient to consider the specific test function $\one_{\M_2(\O)}(g)$ and to integrate it against a matrix coefficient made up of unramified vectors. Our claim is that this data manifests in our theory as the co-multiplication on $E$. Indeed, observe that in the case of $\nu(e)=0$, the co-multiplication $\mu$ corresponds to the distribution:
    \begin{equation*}
        \one_{\M_2(\O)}(g)\delta_1(y).
    \end{equation*}
    
    For more on this, see Remark~\ref{remark:E_is_assoc}. The same idea will also let us define the convolution product on the space of spherical automorphic functions in Section~\ref{sect:T_global}.
\end{remark}

\begin{remark} \label{remark:E_co_mult_is_spin}
    When $\nu(e)\neq 0$, the choice of the matrix
    \[
        \begin{pmatrix}1 & \\ & -\pi^{\nu(e)}\end{pmatrix}
    \]
    in the definition of the co-unit $\varepsilon\co E\ra\one_\Ydown$ in Construction~\ref{const:E_co_unit} seems fairly arbitrary. Indeed, any matrix $h\in\GL_2(F)$ satisfying
    \[
        hKh^{-1}\cap U=\ker(\theta)
    \]
    would have worked after appropriate modifications to $\mu\co E\ra E\oY E$ in Construction~\ref{const:E_co_mult}. This data, together with the choice of the additive character $e\co F\ra\CC^\times$ defining $\theta$, is related to something called an \emph{unramified spin structure}. The precise details of this relationship will be explored in Appendix~\ref{app:spin_struct}.
\end{remark}

\begin{remark}
  The co-multiplication $\mu\co E\ra E\oY E$ is symmetric. That is, the diagram
    \[\xymatrix{
        E \ar[r]^-{\mu} \ar[dr]_{\mu} & E\oY E \ar[d]^{\ref{remark:Y_is_commutative}} \\
        & E\oY E
    }\]
    commutes.  
\end{remark}

We now show that $\mu\co E\ra E\oY E$ and $\varepsilon\co E\ra\one_\Ydown$ are consistent with being the co-multiplication and co-unit of a co-algebra. As discussed above, the following claim immediately implies Theorem~\ref{thm:oY_sph_is_over_center}.
\begin{claim} \label{claim:spherical_is_idemp}
    The diagram
    \begin{equation} \label{eq:co_mult_of_E_is_co_unital} \xymatrix{
        E \ar[d]_{\id_E} \ar[r]^-\mu & E\oY E \ar[d]^{\mathrm{id}_E\oY\varepsilon} \\
        E \ar[r]^-\sim & E\oY\one_\Ydown
    }\end{equation}
    commutes. Moreover, the map $\mu\co E\ra E\oY E$ is an isomorphism.
\end{claim}
\begin{remark}
    It does not yet make sense to ask about the co-associativity of $\mu$, since we have not assigned associativity data to $\oY$. Once we do so in Subsection~\ref{subsect:assoc_Y_from_non_sense} below, we will be able to upgrade Claim~\ref{claim:spherical_is_idemp} and show that $E$ is a co-commutative co-unital co-algebra. See also Remark~\ref{remark:E_is_assoc}.
\end{remark}

\begin{proof}[Proof of Claim~\ref{claim:spherical_is_idemp}]
    In order to prove that Diagram~\eqref{eq:co_mult_of_E_is_co_unital} commutes, it suffices to evaluate it for the uniform distribution of volume $1$ on $\SL_2(\O)$. Under $\mu$, this distribution maps into
    \[
        \Psi(g,y)=q^{-\nu(e)}\cdot\one_{\M_2(\O)}(g)\delta_{\pi^{-\nu(e)}}(y).
    \]
    We wish to apply the map $\varepsilon$, followed by the isomorphism of Construction~\ref{const:one_Y_is_unit}, and then compare the result with the delta distribution on $\SL_2(\O)$, from which we started. We do so via the formula in Equation~\eqref{eq:explicit_unitality}; plugging $\Psi$ and $\varepsilon$ in, we must show that:
    \[
        \int_{\SL_2(\O)}\left(w\begin{pmatrix}1 & \\ & -\pi^{\nu(e)}\end{pmatrix}^{-T}k^{-T}\cdot\Psi\right)\left(g^{-1},-\det(g)\right)\dtimes{k}
    \]
    is the uniform distribution of volume $1$ on $\SL_2(\O)$.
    
    However, because $\Psi$ is $\SL_2(\O)$-invariant, the above simplifies to:
    \begin{multline*}
        \left(\begin{pmatrix}-\pi^{-\nu(e)} & \\ & 1\end{pmatrix}\cdot\Psi\right)\left(g^{-1},-\det(g)\right)= \\
        =\one_{\M_2(\O)}(g^{-1})\delta_{\pi^{-\nu(e)}}(\pi^{-\nu(e)}\det(g))=\one_{\M_2(\O)}(g^{-1})\delta_{1}(\det(g)),
    \end{multline*}
    and we deduce that this is indeed the desired distribution from
    \[
        \M_2(\O)\cap\SL_2(F)=\SL_2(\O).
    \]
    
    It remains to show that $\mu$ is an isomorphism. Let us start by noting that applying the functor $\Phi^-_\mathrm{enh}$ to $\mu$, we obtain a retract of $\Z_\sph$-modules
    \[\xymatrix{
        \Phi^-_\mathrm{enh}E \ar[dr]_{=} \ar[r]^-{\Phi^-_\mathrm{enh}(\mu)} & I_{RL}(E)\otimes_G E \ar[d]^{\Phi^-_\mathrm{enh}(E\oY\varepsilon)} \\
        & \Phi^-_\mathrm{enh}E.
    }\]
    Since $I_{RL}(E)\otimes_G E\cong\Z_\sph$ has no non-trivial retracts when restricted to any component, we deduce that $\Phi^-(\mu)$ is an isomorphism.
    
    Having shown that $\Phi^-$ takes the map $\mu$ into an isomorphism, we see that the co-kernel of $\mu$ is degenerate. However, this co-kernel must be split (Diagram~\eqref{eq:co_mult_of_E_is_co_unital} splits it), and it is easy to see that $Y$ has no degenerate sub-representations. Indeed, because $Y$ consists of smooth and compactly supported functions on $\M_2(F)\times F^\times$, it is clear that it has no vector invariant under the action of $U$.
\end{proof}

We can also combine Theorem~\ref{thm:oY_gen_is_over_center} and Theorem~\ref{thm:oY_sph_is_over_center} as follows. Consider the category
\[
    \Mod(\Z)\times\Mod(\Z_\sph)
\]
whose objects are pairs $(M,N)$ of smooth $\Z$-modules $M$ and smooth $\Z_\sph$-modules $N$. Recall that we are thinking of $\Z_\sph$ as a $\Z$-module. This category admits a unique colimit preserving symmetric monoidal structure which respects the action of $\Z$, whose unit is $(\Z,0)$, and such that $(0,\Z_\sph)$ is idempotent. This symmetric monoidal structure is explicitly defined by
\begin{align*}
    (M_0,N_0)\otimes(M_1,N_1)\ra(M_0\otimes_\Z M_1,\quad & M_0\otimes_\Z N_1 \\
    \oplus & N_0\otimes_\Z M_1 \\
    \oplus & N_0\otimes_{\Z}N_1).
\end{align*}
We now conclude that
\begin{corollary} \label{cor:oY_on_deg_free_sph}
    The faithful functor
    \[
        \Mod(\Z)\times\Mod(\Z_\sph)\ra\Mod(G)
    \]
    given by
    \[
        (M,N)\mapsto (M\otimes_\Z\one_\Ydown)\oplus (N\otimes_{\Z_\sph}E)
    \]
    is weak symmetric monoidal.
\end{corollary}

\subsection{Associativity of \texorpdfstring{$\oY$}{Y}} \label{subsect:assoc_Y_from_non_sense}

In this subsection, we will finally prove that $\oY$ uniquely extends to a symmetric monoidal structure.

Our main theorem is thus:
\begin{theorem} \label{thm:oY_assoc}
    There is a unique extension of $\oY$ from a weak symmetric monoidal structure to a symmetric monoidal structure on $\Mod(G)$.
\end{theorem}

\begin{remark}
    In the previous subsection, we made some choice of isomorphism $E\cong E\oY E$. Because this was a non-canonical choice (see Remark~\ref{remark:E_co_mult_is_spin}), we had some ugly expressions. However, while in this subsection we will need to use the \emph{existence} of such an isomorphism $E\cong E\oY E$, our extension of $\oY$ to a symmetric monoidal structure will be unique, and in particular it will \emph{not} depend on the choice of isomorphism. In other words, the effects of this choice are confined to Subsection~\ref{subsect:Y_for_sph_reps}, with the exception of Remarks~\ref{remark:E_is_assoc} and~\ref{remark:E_comult_to_conv_prod}, which are not a part of the proof of Theorem~\ref{thm:oY_assoc}.
\end{remark}

Before proving Theorem~\ref{thm:oY_assoc}, let us present some nice consequences of our constructions.
\begin{remark} \label{remark:oY_is_frob}
    The symmetric monoidal structure $\oY$ on $\Mod(G)$ is a part of a slightly more sophisticated higher-categorical structure called a \emph{commutative Frobenius algebra} in $\CC$-linear presentable categories.
    
    Specifically, the functor $\Phi^-\co\Mod(G)\ra\Vect$ is a trace map on $\Mod(G)$. This is because its composition with the multiplication:
    \[
        \Phi^-\circ\oY\co\Mod(G)\otimes\Mod(G)\ra\Vect
    \]
    is the pairing $M\boxtimes N\mapsto I_{RL}(M)\otimes_G N$ (see Remark~\ref{remark:phi_minus_is_trace}), which is a self-duality on $\Mod(G)$ in the following higher-categorical sense.
    
    The $2$-category of $\CC$-linear presentable categories with colimit preserving functors is a symmetric monoidal $2$-category, whose product is given by the Lurie tensor product. In this context, $\Phi^-\circ\oY$ is a bi-linear pairing from the category $\Mod(G)$ to the unit $\Vect$ of the symmetric monoidal structure, and it is easy to see that this is in fact a perfect pairing.
\end{remark}

\begin{example} \label{example:irr_reps_are_algs}
    We can now give a much cleaner re-formulation and proof of Theorem~\ref{thm:middle_action}. For any object $V\in\Mod(G)$, we have a canonical morphism:
    \[
        V\otimes\Hom_G(\one_\Ydown,V)\cong V\oY\one_\Ydown\otimes\Hom_G(\one_\Ydown,V)\ra V\oY V.
    \]
    Theorem~\ref{thm:middle_action} can be simply stated as saying that for a generic irreducible representation $V$, the above is an isomorphism. Note that this implies that after a choice of Whittaker model (i.e., an isomorphism $\Hom_G(\one_\Ydown,V)\cong\CC$), we have $V\cong V\oY V$.
    
    In fact, something even stronger can be said: after fixing a vector $v^\gen$ in the one-dimensional space 
    \[
        \Phi^-V=\Hom_G(\one_\Ydown,V),
    \]
    a generic irreducible representation $V$ canonically becomes a unital commutative algebra with respect to $\oY$.
    
    Moreover, this strengthening of Theorem~\ref{thm:middle_action} now has an exceedingly short proof. Indeed, it immediately follows from the following general principle. Let $\C$ be a symmetric monoidal abelian category with a right exact tensor product. Then all surjections $\one_\C\twoheadrightarrow X$ from the unit induce isomorphisms $X\xrightarrow{\sim} X\otimes X$, and moreover turn $X$ into a commutative unital algebra.
\end{example}

\begin{remark} \label{remark:E_is_assoc}
    Using the Idempotent Theorem of \cite{cat_weil}, it is an easy observation that Theorem~\ref{thm:oY_assoc} and Claim~\ref{claim:spherical_is_idemp} imply that $E$ is necessarily co-associative. That is, $E$ is a co-unital co-commutative co-algebra with respect to $\oY$.
    
    This has the following implication. Fix a choice of co-multiplication on $E$ as in Remark~\ref{remark:E_co_mult_is_spin} (recall that this data can be derived from a global spin structure, a fact explored in detail in Appendix~\ref{app:spin_struct}). Let $V$ be any unital and commutative algebra with respect to $\oY$. Then the space $V_{/\SL_2(\O)}=\sHom(E,V)$ of co-invariants is automatically a quasi-unital commutative algebra over $\Z_\sph$, with respect to the \emph{usual tensor product $\otimes$} (note that invariants and co-invariants with respect to $\SL_2(\O)$ are the same, since the group is compact). In the global setting, this will let us define a multiplicative structure on a certain space of spherical automorphic functions. See Appendix~\ref{app:spin_struct} for details.
    
    Informally, the product of $V_{/\SL_2(\O)}$ comes from the product of $V$ using the ``distinguished test function'' appearing in the definition of the co-multiplication of $E$ in Construction~\ref{const:E_co_mult}. That is, to multiply two vectors of $V_{/\SL_2(\O)}$, one tensors them with the distinguished test function, and then applies the product of $V$. See also Remark~\ref{remark:E_co_mult_is_dist_test_func}.
\end{remark}

\begin{remark} \label{remark:E_comult_to_conv_prod}
    We can say something even stronger. As in Remark~\ref{remark:E_co_mult_is_spin}, fix a choice of co-multiplication on $E$, and let $V$ be any unital and commutative algebra with respect to $\oY$, as above. Note that $\one_\Ydown$ is also a co-algebra. Thus, the space $\Phi^-V=\sHom(\one_\Ydown,V)$ can also be automatically upgraded into a quasi-unital commutative algebra over the center $\Z$ (with respect to the usual tensor product $\otimes$). Moreover, the following diagram of algebras commutes:
    \[\xymatrix{
        \Z \ar[r] \ar[d] & \Phi^-V \ar[d] \\
        \Z_\sph \ar[r] & V_{/\SL_2(\O)}.
    }\]
    
    For $V$ a generic irreducible representation as in Example~\ref{example:irr_reps_are_algs}, this is not very interesting, as $\Phi^-V$ and $V_{/\SL_2(\O)}$ are both at most one-dimensional. However, this will be a much more interesting statement for $V=\cI$, the algebra of automorphic functions developed in Section~\ref{sect:T_global}. See also Appendix~\ref{app:spin_struct}.
\end{remark}

The rest of this subsection is dedicated to the proof of Theorem~\ref{thm:oY_assoc}. We begin with the following strengthening of Corollary~\ref{cor:oY_on_deg_free_sph}.
\begin{claim}
    There is at most one extension of $\oY$ from a weak symmetric monoidal structure to a symmetric monoidal structure on $\Mod(G)$. Moreover, the functor
    \[
        \Mod(\Z)\times\Mod(\Z_\sph)\ra\Mod(G)
    \]
    of Corollary~\ref{cor:oY_on_deg_free_sph} is symmetric monoidal for any such extension.
\end{claim}
\begin{proof}
    Consider any extension of $\oY$ to a symmetric monoidal structure (that is, suppose that there is some associator for $\oY$ satisfying the axioms of a symmetric monoidal structure).
    
    We want to show that any two such extensions are equal. Because $\oY$ respects colimits, it is enough to compare their associators
    \[
        (A\oY B)\oY C\xrightarrow{\sim} A\oY(B\oY C)
    \]
    on objects $A,B,C$ that are either of the two generators $\one_\Ydown,E$. Note that if any of the objects $A,B,C$ is equal to $\one_\Ydown$, then the associator is uniquely determined by unitality.
    
    Thus, it remains to show that the only possible choice for the associator
    \[
        (E\oY E)\oY E\xrightarrow{\sim} E\oY(E\oY E)
    \]
    is the one coming from the functor:
    \[
        \Mod(\Z)\times\Mod(\Z_\sph)\ra\Mod(G).
    \]
    That is, it remains to show that the only possible choice for the associator
    \[
        (E\oY E)\oY E\xrightarrow{\sim} E\oY(E\oY E)
    \]
    is the identification $(E\oY E)\oY E\cong E\cong E\oY(E\oY E)$ coming from repeated applications of Claim~\ref{claim:spherical_is_idemp}.
    
    Indeed, in this case the pentagonal axiom
    \[\xymatrix{
        ((E\oY E)\oY E)\oY E \ar[d]^\sim \ar[r]^\sim & (E\oY E)\oY (E\oY E) \ar[r]^\sim & E\oY (E\oY (E\oY E)) \\
        (E\oY (E\oY E))\oY E \ar[rr]^\sim & & E\oY ((E\oY E)\oY E) \ar[u]^\sim
    }\]
    fully determines the associator; any such associator must be given by a map $a\co E\xrightarrow{\sim} E$ (after using Claim~\ref{claim:spherical_is_idemp} to identify $(E\oY E)\oY E\cong E\cong E\oY(E\oY E)$), which must then satisfy:
    \[
        a^2=a^3.
    \]
    Note that this is the same argument as the Idempotent Theorem (Theorem~3.1.2) of \cite{cat_weil}.
\end{proof}

\begin{proof}[Proof of Theorem~\ref{thm:oY_assoc}]
    The functor:
    \[
        \Mod(\Z)\times\Mod(\Z_\sph)\ra\Mod(G)
    \]
    defines associator isomorphisms
    \[
        a_{A,B,C}\co(A\oY B)\oY C\xrightarrow{\sim} A\oY(B\oY C)
    \]
    for $A,B,C$ inside its essential image, by considering the images of the associators in its domain. Note that because $\Mod(\Z)\times\Mod(\Z_\sph)\ra\Mod(G)$ is not fully faithful, this family of isomorphisms is not necessarily natural.
    
    Let $\C_0\subseteq\Mod(G)$ be the fully faithful subcategory given by the essential image of $\Mod(\Z)\times\Mod(\Z_\sph)\ra\Mod(G)$. Because $\oY$ respects colimits, it is uniquely determined by its restriction to $\C_0$; indeed, by standard categorical arguments ($\Mod(G)$ is generated by sifted colimits of compact projective objects, and $\C_0$ contains its compact projective generators and is closed under finite direct sums), it is given by the \emph{left Kan extension}:
    \[
        A\oY B=\colim_{\substack{A_0\in\C_0 \\ \alpha\co A_0\ra A}}\colim_{\substack{B_0\in\C_0 \\ \beta\co B_0\ra B}}A_0\oY B_0.
    \]
    Thus, it is sufficient to show that the associators $a_{A,B,C}$ commute with the morphisms of $\Mod(G)$ on the essential image of $\Mod(\Z)\times\Mod(\Z_\sph)$, and therefore induce a natural morphism on $\C_0$. The required compatibilities with the rest of the symmetric monoidal structure (and the pentagonal axiom) will follow immediately from those of $\Mod(\Z)\times\Mod(\Z_\sph)$.
    
    Because $\Mod(\Z)\times\Mod(\Z_\sph)$ is generated under colimits by $\one_\Ydown$ and $E$, it is sufficient to prove that for all $A,B,C,D\in\{\one_\Ydown,E\}$ and all maps $f\co C\ra D$, the following diagram commutes:
    \[\xymatrix{
        (A\oY B)\oY C \ar[d]^{a_{A,B,C}} \ar[rr]^{(A\oY B)\oY f} & & (A\oY B)\oY D \ar[d]^{a_{A,B,D}} \\
        A\oY (B\oY C) \ar[rr]^{A\oY (B\oY f))} & & A\oY (B\oY D).
    }\]
    
    The claim follows immediately if either $A$ or $B$ is $\one_\Ydown$. Additionally, because $\oY$ respects the action of the center, and because both $\End(\one_\Ydown)$ and $\End(E)$ are covered by the center, it is enough to consider the case $C\neq D$, so either $(A,B,C,D)=(E,E,E,\one_\Ydown)$ or $(A,B,C,D)=(E,E,\one_\Ydown,E)$.
    
    Since the proof in either case is identical, it is enough to consider the case $(A,B,C,D)=(E,E,\one_\Ydown,E)$. Letting $f\co\one_\Ydown\ra E$, we must show that the following diagram commutes:
    \[\xymatrix{
        & (E\oY E)\oY \one_\Ydown \ar[rr]^{(E\oY E)\oY f} & & (E\oY E)\oY E & \\
        E\oY E \ar[ru]_{\cong} \ar[rd]^{\cong} & & & & E\oY E \ar[ld]_{\cong} \ar[lu]^{\cong}\\
        & E\oY (E\oY \one_\Ydown) \ar[rr]^{E\oY (E\oY f))} & & E\oY (E\oY E). &
    }\]
    This is equivalent to asking that the two maps
    \[
        \Hom(\one_\Ydown,E)\ra\Hom(E,E)
    \]
    given by $-\oY E$ and $(-\oY E)\oY E$ be equal.
    
    In other words, it is sufficient to prove that the map
    \[
        \Hom(E,E)\ra\Hom(E,E)
    \]
    given by $-\oY E$ is the identity map of $\Hom(E,E)$. But the center of $\Mod(G)$ covers $\Hom(E,E)$. Therefore, it is enough to test this property on the identity map of $E$, and verify that
    \[
        \id_E=\id_E\oY E.
    \]
    However, this holds by functoriality of $\oY$.
\end{proof}

\part{Abstractly Automorphic Representations} \label{part:abst_aut_cat}

\section{Global Automorphic Forms} \label{sect:T_global}

\subsection{Introduction}

Our goal for this section is to observe that an appropriate space of automorphic functions becomes a commutative algebra under the symmetric monoidal structure $\oY$.

Let $F$ be a global function field of characteristic $\neq 2$. Fix the standard Haar measure $\dtimes{g}$ on $\AA^\times$ (respectively, $\GL_2(\AA)$) such that the product of the maximal compact subgroups $\O_v^\times$ (respectively, $\GL_2(\O_v)$) is of measure $1$.

Consider the space
\[
    \cS=S(\GL_2(F)\backslash\GL_2(\AA))
\]
of automorphic functions -- smooth and compactly supported functions on the automorphic quotient $\GL_2(F)\backslash\GL_2(\AA)$. Let
\[
    p\co\cS\ra S(\AA^\times/F^\times)
\]
be the projection along the determinant map
\[
    \det\co\GL_2(\AA)\ra\AA^\times,
\]
and let $\cI$ be the kernel of $p$ (note that $\cI$ is independent of the choices of Haar measures). Given any two functions $\phi,\phi'\in\cS$ and a test function $\Psi\in S(\M_2(\AA))$, one obtains a Godement--Jacquet zeta integral
\begin{multline*}
    Z_\text{GJ}\left(\phi',\phi,\Psi,s\right)= \\
    \int_{\GL_2(\AA)}\left(\int_{\GL_2(F)\backslash\GL_2(\AA)}\phi'(h^{-T})\phi(hg)\dtimes{h}\right)\Psi(g)\abs{\det(g)}^{s+\frac{1}{2}}\dtimes{g}.
\end{multline*}
On the other hand, a single function $\phi''\in\cS$ with Whittaker function $W_{\phi''}$ also has an associated Jacquet--Langlands zeta integral
\[
    Z_\text{JL}\left(\phi'',s\right)=\int_{\AA^\times}W_{\phi''}\left(\begin{pmatrix}y & \\ & 1\end{pmatrix}\right)\abs{y}^{s-\frac{1}{2}}\dtimes{y}.
\]

Below, we will see that the equality of Godement--Jacquet and Jacquet--Langlands $L$-functions more or less induces a multiplication map on $\cS$. More precisely, we will obtain a map
\begin{equation*}
    \cI\oY\cI\ra\cI
\end{equation*}
on the kernel $\cI$ of $p\co\cS\ra S(\AA^\times/F^\times)$ which turns Godement--Jacquet zeta integrals into Jacquet--Langlands zeta integrals. We will show that this multiplication map makes $\cI$ into a commutative unital algebra \emph{with respect to the symmetric monoidal structure $\oY$}.

To put this construction in context inside the theory of automorphic representations, one can think about the multiplication map we construct as a kind of theta lifting of pairs of automorphic forms. This is essentially a shadow of the construction of $Y$ in Section~\ref{sect:GJ_vs_JL} as a Weil representation.

Before delving into the structure of this section, let us detail some more of the remarkable properties of the algebra $\cI$. For starters, it turns out that the unit map
\[
    e\co\one_\Ydown\ra\cI
\]
is surjective. This implies that while the multiplication map uniquely determines the unit, the other direction holds as well. In fact, for \emph{any} smooth $\GL_2(\AA)$-module $V$, being an $\cI$-module is a mere property, equivalent to the existence of a (necessarily unique) lift
\[\xymatrix{
    \one_\Ydown\oY V \ar[r] \ar[d]_{e\oY\id} & V \\
    \cI\oY V. \ar@{-->}[ru] &
}\]
It turns out that $\cS$, along with all irreducible automorphic representations, are $\cI$-modules. This justifies thinking of the property of being an $\cI$-module as an automorphicity property. In fact, it appears that the category $\Mod^\aut(\GL_2(\AA))$ of $\cI$-modules is a natural context for the study of automorphic representations.

\begin{remark} \label{remark:mult_is_like_conv}
    The multiplication $\cI\oY\cI\ra\cI$ continues the theme discussed in Remark~\ref{remark:oY_is_like_comm}, where the symmetric monoidal structure $\oY$ seems to mimic the behaviour of the relative tensor product over a commutative group. That is, over a commutative group $H$, the space of functions on a quotient $\Gamma\backslash H$ of the group carries a convolution product. Intuitively, it seems as though the multiplication on $\cI$ takes the same role as this convolution product, despite the group $\GL_2(\AA)$ being non-commutative. The cost of this non-commutativity is the use of the exotic symmetric monoidal structure $\oY$.
\end{remark}

One can think of the structure of this section as one big construction of the algebra structure on $\cI$. Let us be more specific.

We start in Subsection~\ref{subsect:glob_oY}, where we fix some definitions for $\oY$ in the global case, as we have only discussed $\oY$ locally so far.

In Subsection~\ref{subsect:aut_Y_mult}, we will define the desired multiplication map. In fact, we will construct a single multiplication map $m\co\cS\oY\cS\ra\widetilde{\cS}$ into the contragradient $\widetilde{\cS}$ of $\cS$ (which canonically contains $\cS$). At this point, the multiplication will simply be a map; we will not even know if the subspace $\cI\subseteq\cS$ is closed under multiplication.

In Subsection~\ref{subsect:aut_Y_unit}, we will show that this multiplication is unital by explicitly constructing a unit map $e\co\one_\Ydown\ra\cI$. In Subsection~\ref{subsect:aut_Y_assoc}, we will show that this unit map $e\co\one_\Ydown\ra\cI$ is surjective. This will imply that the multiplication we defined is associative, as well as show that $\cI$ is closed under multiplication, and establish the desired algebraic structure. Finally, in Subsection~\ref{subsect:Y_alg_sheaves}, we will discuss the property of being an $\cI$-module, and construct the category of abstractly automorphic representations.

Having established our desired constructions, in Appendix~\ref{app:spin_struct} we will fulfill a debt and show how the co-algebra $E$ of Subsection~\ref{subsect:Y_for_sph_reps} is related to global spin structures. This will allow us to show that the spherical functions in $\cI$ are an algebra with respect to the usual tensor product $\otimes$, as described in Section~\ref{sect:intro_to_triality}.

\subsection{Global Symmetric Monoidal Structure} \label{subsect:glob_oY}

We begin by turning the local symmetric monoidal structures $\oY$ at each place $F_v$ into a global symmetric monoidal structure $\oY$ on the category $\Mod(\GL_2(\AA))$ of smooth $\GL_2(\AA)$-modules.

Define
\[
    {Y_\AA}=S(\M_2(\AA)\times\AA^\times)
\]
via a restricted tensor product of the local spaces $S(\M_2(F_v)\times F_v^\times)$, with respect to the distinguished functions $\one_{\M_2(\O_v)}(g)\one_{\O_v^\times}(y)$. Fix a non-trivial additive character $e\co\AA/F\ra\CC^\times$ once and for all. Using $e$, Construction~\ref{const:middle_action} gives the space ${Y_\AA}$ a smooth $\GL_2(\AA)^3$-action. Note that in this section, we will allow ourselves to more implicitly switch between right and left $\GL_2(\AA)$ actions. However, we will always do so via the transposition map, as in Definition~\ref{def:directions_for_G_modules}.

We define the bi-functor
\[
    \oY\co\Mod(\GL_2(\AA))\times\Mod(\GL_2(\AA))\ra\Mod(\GL_2(\AA))
\]
by
\[
    V\oY V'=V\otimes_G {Y_\AA}\otimes_G V'.
\]
Here, $\otimes_G$ denotes the relative tensor product over the global group $G=\GL_2(\AA)$.

The global unit $\one_\Ydown$ is defined as the $\theta$-co-invariants of $S(\GL_2(\AA))$ with respect to the right action of $U(\AA)$, where $\theta\co U(\AA)/U(F)\ra\CC^\times$ is given by
\[
    \theta\left(\begin{pmatrix}1 & u \\ & 1\end{pmatrix}\right)=e(u).
\]
\begin{remark} \label{remark:one_Y_is_comp_ind_global}
    As in Remark~\ref{remark:one_Y_is_comp_ind} in the local case, fixing a Haar measure on $U=U(\AA)$ allows us to define a canonical isomorphism:
    \[
        \one_\Ydown\xrightarrow{\sim}\cInd_U^G(\theta),
    \]
    where $\cInd_U^G(\theta)$ is the representation of $G$ obtained from $\theta$ by induction with compact support.
\end{remark}

We now claim that the local commutativity, associativity and unitality data turns $\oY$ into a symmetric monoidal structure on $\Mod(\GL_2(\AA))$. This requires some compatibility of the commutativity, associativity and unitality data of Section~\ref{sect:T_tensor} with the distinguished vector $\one_{\M_2(\O_v)}(g)\one_{\O_v^\times}(y)$ at each place:
\begin{claim}
    For a place $v$ of $F$, let $Y_v=S(\M_2(F_v)\times F_v^\times)$, let $G_v=\GL_2(F_v)$, and let $y_v=\one_{\M_2(\O_v)}(g)\one_{\O_v^\times}(y)\in Y_v$ be the distinguished vector. The following hold for almost all places $v$:
    \begin{enumerate}
        \item \label{item:distinguished_unitality} The isomorphism $\one_\Ydown\otimes_{G_v}Y_v\cong S(G_v)$ of Construction~\ref{const:one_Y_is_unit} sends distinguished vectors to distinguished vectors.
        \item \label{item:distinguished_commutativity} The action of the permutation $(1,3)$ interchanging the left and right actions on $Y_v$ sends the distinguished vector $y_v\in Y_v$ to itself.
        \item \label{item:distinguished_associativity} The associativity constraint $Y_v\otimes_{G_v} Y_v\cong Y_v\otimes_{G_v} Y_v$ sends the distinguished vector $y_v\otimes y_v$ to itself.
    \end{enumerate}
\end{claim}
\begin{proof}
    Recall that for every place $v$, the space $E_v=\cInd_{\SL_2(\O_v)}^{G_v}(\one_{\SL_2(\O_v)})$ is a co-commutative, co-unital, co-algebra by Remark~\ref{remark:E_is_assoc}. For places $v$ where $e_v$ is unramified, denote by $f_v$ the vector $\one_{\GL_2(\O)}\in E_v$.
    
    We observe that under the co-unit $E_v\ra\one_{\Ydown,v}$ of Construction~\ref{const:E_co_unit}, the vector $f_v$ maps to the distinguished vector $\one_{\GL_2(F_v)}$ of $\one_{\Ydown,v}$. Moreover, observe that under the co-multiplication map $E_v\ra E_v\oY E_v$ of Construction~\ref{const:E_co_mult}, the vector $f_v$ maps to $f_v\otimes y_v\otimes f_v\in E_v\otimes_{G_v}Y_v\otimes_{G_v} E_v$.
    
    Therefore, the co-unitality of $E_v$ implies Item~\ref{item:distinguished_unitality}, the co-commutativity of $E_v$ implies Item~\ref{item:distinguished_commutativity}, and the co-associativity of $E_v$ implies Item~\ref{item:distinguished_associativity}.
\end{proof}

\begin{corollary}
    The local commutativity, associativity and unitality data comprise well-defined global maps, turning $\oY$ into a symmetric monoidal structure on $\Mod(\GL_2(\AA))$.
\end{corollary}

\subsection{The \texorpdfstring{$\oY$}{Y}-Algebra of Automorphic Forms} \label{subsect:aut_Y_mult}

In this subsection, we will define the promised multiplication map
\[
    m\co\cS\oY\cS\ra\widetilde{\cS}.
\]
Later, in Subsection~\ref{subsect:aut_Y_assoc}, we will show an associativity-type statement for it, as well as prove that the image of $\cI\oY\cI$ lies inside $\cI$. This will turn $\cI$ into an algebra.

Let us begin with an informal discussion. We are trying to construct a map
\[
    m\co\cS\oY\cS=\cS\otimes_G {Y_\AA}\otimes_G\cS\ra\widetilde{\cS}
\]
which sends Godement--Jacquet zeta integrals to Jacquet--Langlands zeta integrals. This map should take as input pairs of automorphic functions, as well as a test function from ${Y_\AA}$, and yield an element of $\widetilde{\cS}$. We think of the co-domain $\widetilde{\cS}$ as the space of smooth functions on $\GL_2(F)\backslash\GL_2(\AA)$, without any conditions on growth. The Godement--Jacquet zeta integral of the input data (two automorphic functions and a test function) should be the same as the Jacquet--Langlands zeta integral of the output (a smooth function on $\GL_2(F)\backslash\GL_2(\AA)$). We will formalize this compatibility property in Remark~\ref{remark:cI_mult_GJ_vs_JL}.

This compatibility with zeta integrals should in principle be enough to recover the multiplication map $m$, for the following reason. Recovering a sufficiently nice function $\phi''\in\widetilde{\cS}$ from its family of Jacquet--Langlands zeta integrals is fairly standard. One simply integrates the zeta integrals over all unitary characters $\chi\co\AA^\times/F^\times\ra\CC^\times$ (i.e., applies the inverse Mellin transform) to isolate $\phi''(1)$. The rest of the values of $\phi''$ can be recovered by the $\GL_2(\AA)$-equivariance of the map we are trying to construct.

Let $\prescript{}{\GL_2(F)\backslash}{{Y_\AA}}_{/\GL_2(F)}$ denote the co-invariants of ${Y_\AA}$ with respect to the left and right $\GL_2(F)$-actions. The above discussion suggests defining a linear functional
\[
    \prescript{}{\GL_2(F)\backslash}{{Y_\AA}}_{/\GL_2(F)}=\cS\otimes_G {Y_\AA}\otimes_G\cS\ra\CC
\]
that should informally correspond to taking the Godement--Jacquet zeta integral of the input, and applying the inverse Mellin transform to it. If we did everything correctly, it will then be invariant under the middle action of $\GL_2(F)$. With this extra invariance property, we will be able to equivariantly extend the map $\cS\otimes_G {Y_\AA}\otimes_G\cS\ra\CC$ into the desired multiplication map $\cS\otimes_G {Y_\AA}\otimes_G\cS\ra\widetilde{\cS}$.

In other words, due to the universal property of the contragradient, it is enough to give a linear functional
\[
    \mu\co {Y_\AA}\ra\CC
\]
which is invariant under all three actions of $\GL_2(F)$. Note that this goal is inherently symmetric in the three actions.

The structure of the rest of this subsection is as follows. We will formally define the desired linear functional in Subsubsection~\ref{subsubsect:aut_Y_mult_def}. In Subsubsection~\ref{subsubsect:aut_Y_mult_inv_functional}, we will show that this functional is invariant under the full action of $\GL_2(F)^3$. Finally, in Subsubsection~\ref{subsubsect:aut_Y_mult}, we will define the multiplication map $m\co\cS\oY\cS\ra\widetilde{\cS}$ itself, and discuss its compatibility with zeta integrals.

\subsubsection{Defining the Functional} \label{subsubsect:aut_Y_mult_def}

We propose the following functional, which will be justified by Remark~\ref{remark:cI_mult_GJ_vs_JL} below.
\begin{definition}
    Define the linear functional
    \begin{align*}
        \mu\co {Y_\AA}\ra\CC
    \end{align*}
    by the formula
    \[
        \mu(\Psi)\mapsto\sum_{(\xi,q)\in\M_2(F)\times F^\times}\Psi(\xi,q).
    \]
\end{definition}
    

\begin{remark} \label{remark:Y_mult_invar_mirabolic}
    The functional $\mu$ is clearly invariant with respect to the left and right actions of $\GL_2(F)$. Moreover, $\mu$ is also invariant with respect to the middle action of $P_2(F)$. This follows using the explicit formula for the middle action of the mirabolic group on ${Y_\AA}$ given in Equation~\eqref{eq:mirabolic_middle_action}.
\end{remark}

\begin{remark}
    The functional $\mu$ is clearly symmetric under the action of the transposition of $\M_2(\AA)$ on ${Y_\AA}$ (see Remark~\ref{remark:Y_is_commutative}).
\end{remark}

\subsubsection{Invariance of the Functional} \label{subsubsect:aut_Y_mult_inv_functional}

Our goal for this subsubsection is to prove that $\mu$ is indeed invariant under all three of the $\GL_2(F)$-actions. This will give us a map
\[
    m\co\cS\oY\cS\ra\widetilde{\cS}
\]
into the smoothening $\widetilde{\cS}$ of the \emph{dual} of $\cS$ (i.e., the contragradient of $\cS)$.

\begin{remark}
    The reader should note that the strategy we are proposing looks like a kind of theta lifting. The image of $\cS\oY\cS$ in $\widetilde{\cS}$ is essentially given by a theta function, up to issues related to the center of $\GL_2$.
\end{remark}

In fact, we prove something even stronger:
\begin{proposition} \label{prop:Y_mult_invar}
    The functional $\mu\co {Y_\AA}\ra\CC$ is invariant with respect to the $S_3\ltimes \GL_2(F)^3$-action of Remark~\ref{remark:Y_is_comm_frob}.
\end{proposition}
\begin{proof}
    This immediately follows from the invariance of the theta functional
    \begin{align*}
        S(\M_2(\AA)) & \ra\CC \\
        \Psi & \mapsto\sum_{\xi\in\M_2(F)}\Psi(\xi)
    \end{align*}
    under the action of the group of rational points of the metaplectic group used in the construction of ${Y_\AA}$. 
    
    Nevertheless, let us explicitly show the invariance under the middle action of $\GL_2(F)$. By Remark~\ref{remark:Y_mult_invar_mirabolic}, it is sufficient to prove that $\mu$ is invariant under the middle action of the matrix $w=\begin{pmatrix} & -1 \\ 1 & \end{pmatrix}$. Using Remark~\ref{remark:middle_w_action}, it is enough to show that:
    \[
        \sum_{(\xi,q)\in\M_2(F)\times F^\times}\Psi(\xi,q)=\sum_{(\xi,q)\in\M_2(F)\times F^\times}\abs{q}^2\int_{\M_2(F)}\Psi(h,q)\cdot e\bigg(-q\cdot\left<\xi,h\right>\bigg)\d{h}.
    \]
    However,
    \begin{multline*}
        \sum_{(\xi,q)\in\M_2(F)\times F^\times}\abs{q}^2\int_{\M_2(F)}\Psi(h,q)\cdot e\bigg(-q\cdot\left<\xi,h\right>\bigg)\d{h}= \\
        =\sum_{(\xi,q)\in\M_2(F)\times F^\times}\int_{\M_2(F)}\Psi(h,q)\cdot e\bigg(-\tr(\xi^T h)\bigg)\d{h},
    \end{multline*}
    where we have used the identity
    \[
        \left<\xi,h\right>=\tr(w^{-1}\xi^T wh).
    \]
    We are now done by the Poisson summation formula:
    \[
        \sum_{(\xi,q)\in\M_2(F)\times F^\times}\Psi(\xi,q)=\sum_{(\xi,q)\in\M_2(F)\times F^\times}\int_{\M_2(F)}\Psi(h,q)\cdot e\bigg(-\tr(\xi^T h)\bigg)\d{h}.
    \]
\end{proof}

\subsubsection{Defining the Multiplication} \label{subsubsect:aut_Y_mult}

Proposition~\ref{prop:Y_mult_invar} justifies the following construction:
\begin{construction} \label{const:aut_mult}
    We define a map $m\co\cS\oY\cS\ra\widetilde{\cS}$ by composing the identification
    \[
        \cS\oY\cS\cong\prescript{}{\GL_2(F)\backslash}{{Y_\AA}}_{/\GL_2(F)}
    \]
    with the map
    \begin{align*}
        \prescript{}{\GL_2(F)\backslash}{{Y_\AA}}_{/\GL_2(F)} & \ra\widetilde{\cS} \\
        \Psi & \mapsto\mu(g\cdot \Psi).
    \end{align*}
    Here, $g\cdot \Psi$ denotes the middle action.
\end{construction}

In our informal discussion at the beginning of this subsection, it was emphasized that the multiplication map $m$ should be compatible with zeta integrals in some sense. Let us formalize this.
\begin{remark} \label{remark:cI_mult_GJ_vs_JL}
    The multiplication $m\co\cS\oY\cS\ra\widetilde{\cS}$ sends Godement--Jacquet zeta integrals into Jacquet--Langlands zeta integrals. Specifically, we claim that:
    \[
        Z_\text{JL}\left(m(\phi'\otimes\Psi\otimes f\otimes\phi),s\right)=Z_\text{GJ}\left(\phi',\phi,\Psi,s\right)\cdot\int_{\AA^\times} f(y)\abs{y}^{s+\frac{1}{2}}\dtimes{y},
    \]
    where $\Psi\otimes f\in {Y_\AA}$, $\phi,\phi'\in\cS$, and
    \begin{align*}
        Z_\text{JL}\left(\phi'',s\right) & =\int_{\AA^\times} W_{\phi''}\left(\begin{pmatrix}y & \\ & 1\end{pmatrix}\right)\abs{y}^{s-\frac{1}{2}}\dtimes{y}, \\
        Z_\text{GJ}\left(\phi',\phi,\Psi,s\right) & =\int_{\GL_2(\AA)} \left(\int_{\GL_2(F)\backslash\GL_2(\AA)}\phi'(h^{-T})\phi(hg) \dtimes{h}\right) \times \\
        & \qquad\qquad\qquad\qquad\qquad\times\Psi(g)\abs{\det(g)}^{s+\frac{1}{2}}\dtimes{g} \\
    \end{align*}
    are the Jacquet--Langlands and Godement--Jacquet zeta integrals, respectively. Here, $W_{\phi''}$ is the Whittaker function corresponding to $\phi''\in\widetilde{\cS}$. This can be easily seen via a direct computation.
\end{remark}

\subsection{Unitality of \texorpdfstring{$\cS$}{S}} \label{subsect:aut_Y_unit}

In the previous subsection, we defined a multiplication map $m\co\cS\oY\cS\ra\widetilde{\cS}$ (Construction~\ref{const:aut_mult}). We have not yet proven that $\cI\oY\cI$ maps to $\cI$, and not yet shown any associativity properties of $m$. In this subsection, we will discuss the \emph{unitality} of the multiplication $m$.

Recall that in a category $\C$ with a symmetric monoidal structure $\otimes$, a commutative algebra object $A\in\C$ is called \emph{unital} if there exists a (necessarily unique) map $e\co\one_\C\ra A$ such that
\begin{equation*}\xymatrix{
    A \ar[r]^-\sim & A\otimes\one_\C \ar[r]^{A\otimes e} & A\otimes A \ar[r] & A
}\end{equation*}
is the identity map.

Specifically, in this subsection, we will begin by constructing a map $e\co\one_\Ydown\ra\cS$ (which, in fact, factors through $\cI$). Afterwards, we will show (Corollary~\ref{cor:mult_on_cS_is_unital}) that the composition
\begin{equation} \label{eq:unitality_of_cS} \xymatrix{
    \cS \ar@{=}[r] & \cS\oY\one_\Ydown \ar[r]^{\cS\oY e} & \cS\oY\cS \ar[r]^-m & \widetilde{\cS}
}\end{equation}
is the inclusion $\cS\subseteq\widetilde{\cS}$ adjoint to the Petersson pairing $\cS\otimes_G\cS\ra\CC$. This is the desired unitality of $m$, which will later translate into $e$ being the unit for the commutative algebra $\cI$ with its multiplication induced by $m$.

\begin{remark}
    We are going to give an explicit construction of the unit map $e$ in this subsection, while in the previous subsection we gave an explicit construction for the multiplication map $m$. However, these two constructions are not independent, and more or less determine each other.
    
    Indeed, it should be noted that as with all algebras, the multiplication map $m$ (given in Construction~\ref{const:aut_mult}) and the unitality property (given in Diagram~\eqref{eq:unitality_of_cS}) are already sufficient to uniquely determine $e$. Moreover, in our case the other direction holds as well; this means the following. We will show in Subsection~\ref{subsect:aut_Y_assoc} that the unit map $e$ has image $\cI\subseteq\cS$. This will mean that the unit map $e$ already uniquely determines the restriction of the multiplication map $m$ to $\cI\oY\cS$ (see also Remark~\ref{remark:e_determines_mult}).
\end{remark}

Let us begin. The cleanest way to describe the map $e\co\one_\Ydown\ra\cS$ is by looking at its dual, $W\co\widetilde{\cS}\ra\widetilde{\one_\Ydown}$:
\begin{construction} \label{const:whitt_map}
    We think of $\widetilde{\cS}$ as the space of smooth functions on the automorphic quotient $\GL_2(F)\backslash \GL_2(\AA)$, and of $\widetilde{\one_\Ydown}$ as the space of $\theta^{-1}$-equivariant functions on $\GL_2(\AA)$. We define the map $W\co\widetilde{\cS}\ra\widetilde{\one_\Ydown}$ by sending an automorphic function to its corresponding Whittaker function:
    \[
        \phi(g)\mapsto W_\phi(g)=\int_{U(\AA)/U(F)}\phi(ug)\theta(u)\d{u}.
    \]
\end{construction}

We now claim that
\begin{claim} \label{claim:unit_contra_to_whittaker}
    There is a unique map $\one_\Ydown\ra\cS$ such that its contragradient is the map $W\co\widetilde{\cS}\ra\widetilde{\one_\Ydown}$ of Construction~\ref{const:whitt_map}.
\end{claim}

Indeed, it is clear that the contragradient is a faithful functor. Therefore, we can prove Claim~\ref{claim:unit_contra_to_whittaker} by directly constructing the desired map $e\co\one_\Ydown\ra\cS$. This will be our unit for the multiplication on the algebra $\cI$.
\begin{construction} \label{const:unit_of_cS}
    Define the map
    \[
        e\co\one_\Ydown\ra\cS
    \]
    as the composition of
    \begin{align*}
        S(\GL_2(\AA))_{/U_2(\AA),\theta} & \ra S(U_2(F)\backslash\GL_2(\AA)) \\
        f(g) & \mapsto\int_{U_2(\AA)}f(ug)\theta(u^{-1})\d{u}
    \end{align*}
    with the projection $S(U_2(F)\backslash\GL_2(\AA))\ra\cS$. In other words, using the equivalence of Remark~\ref{remark:one_Y_is_comp_ind_global}, a function $W\in\cInd_U^G(\theta)\cong\one_\Ydown$ is mapped to the automorphic function
    \[
        \sum_{\gamma\in U_2(F)\backslash\GL_2(F)}W(\gamma g)\in S(\GL_2(F)\backslash\GL_2(\AA))\cong\cS.
    \]
\end{construction}

The map $e\co\one_\Ydown\ra\cS$ of Construction~\ref{const:unit_of_cS} will be a unit with respect to the multiplication $m\co\cS\oY\cS\ra\widetilde{\cS}$ of Construction~\ref{const:aut_mult}. To show this, we will need the following proposition about the Petersson pairing.

Recall that the Petersson pairing is the map $\rho\co\cS\otimes_G\cS\ra\CC$ given by
\[
    f(g)\otimes f'(g')\mapsto\int_{\GL_2(\AA)}f(g)f'(g^{-T})\dtimes{g},
\]
where we have implicitly used $\cS\otimes_G\cS$ instead of the more formally correct $I_{RL}(\cS)\otimes_G\cS$. This can also be described by the functional:
\begin{align*}
    \prescript{}{\GL_2(F)\backslash}{S(\GL_2(\AA))}_{/\GL_2(F)} & \ra\CC \\
    f(g) & \mapsto\sum_{\xi\in\GL_2(F)}f(\xi),
\end{align*}
where the notation $\prescript{}{\GL_2(F)\backslash}{S(\GL_2(\AA))}_{/\GL_2(F)}$ denotes the co-invariants of the vector space $S(\GL_2(\AA))$ under the two actions of $\GL_2(F)$.

\begin{proposition} \label{prop:mult_on_cS_is_unital}
    The composition
    \[
        \cS\otimes_G\cS\cong\one_\Ydown\otimes_G(\cS\oY\cS)\xrightarrow{e\otimes\id}\cS\otimes_G(\cS\oY\cS)\xrightarrow{\mu}\CC
    \]
    is the Petersson pairing $\rho$.
\end{proposition}

\begin{remark}
    Proposition~\ref{prop:mult_on_cS_is_unital} is independent of \emph{which} of the three copies of $\cS$ we insert $\one_\Ydown$ into. This follows by the $S_3$-invariance shown in Proposition~\ref{prop:Y_mult_invar}.
\end{remark}

Finally, the desired unitality property of $e\co\one_\Ydown\ra\cS$ follows because $\mu$ is given by $m$ together with the natural evaluation pairing:
\begin{corollary} \label{cor:mult_on_cS_is_unital}
    The composition
    \[\xymatrix{
        \cS \ar@{=}[r] & \one_\Ydown\oY\cS \ar[r]^{e\otimes\id} & \cS\oY\cS \ar[r]^-m & \widetilde{\cS}
    }\]
    is the inclusion $\cS\subseteq\widetilde{\cS}$.
\end{corollary}

\begin{proof}[Proof of Proposition~\ref{prop:mult_on_cS_is_unital}]
    We want to show that the two maps $\cS\otimes_G\cS\ra\CC$ coincide. This is a straightforward verification. We will show that the diagram
    \[\xymatrix{
        \one_\Ydown\otimes_G\left(\prescript{}{\GL_2(F)\backslash}{{Y_\AA}}_{/\GL_2(F)}\right) \ar[r]^-\sim & \one_\Ydown\otimes_G(\cS\oY\cS) \ar[r]^-{e\otimes\id} \ar[d]^\sim & \cS\otimes_G(\cS\oY\cS) \ar[d]^\mu \\
        & \cS\otimes_G\cS \ar[r]^-\rho & \CC
    }\]
    commutes by separately evaluating the clockwise and counter-clockwise compositions.
    
    We begin with the counter-clockwise composition. Using Equation~\eqref{eq:explicit_unitality}, it is given by:
    \begin{equation*}
        f(h)\otimes\Psi(g,y)\mapsto\sum_{\xi\in\GL_2(F)}\int_{\GL_2(\AA)}(wh^{-T}\cdot\Psi)(\xi^{-1},-\det(\xi))\cdot f(h)\dtimes{h}.
    \end{equation*}
    
    We should compare this with the clockwise composition:
    \begin{multline*}
        f(h)\otimes\Psi(g,y)\mapsto \\
        \mapsto\sum_{(\xi,q)\in\M_2(F)\times F^\times}\int_{U_2(F)\backslash\GL_2(\AA)}\int_{U_2(\AA)}(h^{-T}\cdot\Psi)(\xi,q)\cdot f(uh)\cdot\theta(u^{-1})\d{u}\dtimes{h}= \\
        =\sum_{(\xi,q)\in\M_2(F)\times F^\times}\int_{U_2(F)\backslash\GL_2(\AA)}\int_{U_2(\AA)}(wh^{-T}\cdot\Psi)(\xi,-q)\cdot f(uh)\cdot\theta(u^{-1})\d{u}\dtimes{h}= \\
        =\sum_{(\xi,q)\in\M_2(F)\times F^\times}\int_{\GL_2(\AA)}(wh^{-T}\cdot\Psi)(\xi,-q)\cdot f(h)\cdot \int_{\AA/F}e(x(q\det{\xi}-1))\d{x}\dtimes{h}.
    \end{multline*}
    Thus, the two maps coincide.
\end{proof}

\subsection{Associativity of the Multiplication} \label{subsect:aut_Y_assoc}

We are now armed with a unital multiplication map $m\co\cS\oY\cS\ra\widetilde{\cS}$. Our goal in this subsection will be to prove that it is associative in some sense, and to prove that $\cI$ is an algebra.

Everything will immediately follow from the following proposition:
\begin{proposition} \label{prop:one_Y_onto_cI}
    The image of the unit map $e\co\one_\Ydown\ra\cS$ of Construction~\ref{const:unit_of_cS} is $\cI$.
\end{proposition}
\begin{remark}
    In a very vague sense, Proposition~\ref{prop:one_Y_onto_cI} can be thought of as a multiplicity one property; because a generic irreducible $V$ has $\dim\Phi^-V=1$, it follows that the space of maps between $\cI$ and $V$ is at most one-dimensional as well. Following the proof of Proposition~\ref{prop:one_Y_onto_cI} below, one can also see that it relies on similar ideas to the usual proof of multiplicity one.
\end{remark}

Before proving Proposition~\ref{prop:one_Y_onto_cI}, let us state some of its corollaries.

\begin{corollary} \label{cor:mult_image_of_cI}
    The image $m(\cI\oY\cI)$ lies inside $\cI\subseteq\cS\subseteq\widetilde{\cS}$.
\end{corollary}
\begin{proof}
    Consider the commutative diagram:
    \[\xymatrix{
        \one_\Ydown\oY\cI \ar[r]^-\sim \ar@{->>}[d] & \cI \ar@{^{(}->}[d] \\
        \cI\oY\cI \ar@{-->}[ur] \ar[r]^-m & \widetilde{\cS}.
    }\]
    The left vertical map is surjective, while the right vertical map is injective. This means that the diagram can be completed as above.
\end{proof}

\begin{corollary} \label{cor:cI_alg_cS_module}
    The multiplication $m\co\cS\oY\cS\ra\widetilde{\cS}$ makes $\cI$ into a unital commutative algebra with respect to $\oY$.
\end{corollary}
\begin{proof}
    Follows immediately from the surjectivity of $e\co\one_\Ydown\ra\cI$, and the axioms of symmetric monoidal categories. By Corollaries~\ref{cor:mult_on_cS_is_unital} and~\ref{cor:mult_image_of_cI}, it is sufficient to show that the multiplication $m\co\cI\oY\cI\ra\cI$ is associative. Indeed, we have two maps
    \[\xymatrix{
        \cI\oY\cI\oY\cI \ar@<3pt>[rr]^-{m\circ(\id\oY m)} \ar@<-3pt>[rr]_-{m\circ(m\oY\id)} & & \cI,
    }\]
    and we want to show that they are the same. By Proposition~\ref{prop:one_Y_onto_cI}, we can test this by pre-composing with the unit map $e\co\one_\Ydown\ra\cI$ at one of the coordinates. However, the two maps
    \[\xymatrix{
        \one_\Ydown\oY\cI\oY\cI \ar[r] & \cI\oY\cI\oY\cI \ar@<3pt>[rr]^-{m\circ(\id\oY m)} \ar@<-3pt>[rr]_-{m\circ(m\oY\id)} & & \cI
    }\]
    are the same by the unitality result of Corollary~\ref{cor:mult_on_cS_is_unital}.
\end{proof}

\begin{remark} \label{remark:e_determines_mult}
    In fact, the surjectivity of the unit map implies that $m$ is the \emph{unique} multiplicative structure on $\cI$ with respect to this unit.
\end{remark}

\begin{remark}
    Observe that $\cI$ becomes not just an algebra, but it also carries a pairing
    \[
        \left<-,-\right>\co\cI\otimes_G\cI=\Phi^-(\cI\oY\cI)\ra\CC
    \]
    induced from the Petersson pairing $\rho\co\cS\otimes_G\cS\ra\CC$ (here, $\Phi^-\co\Mod(G)\ra\Vect$ indicates the functor of $\theta$-co-invariants; recall that $\Phi^-$ defines a trace on the category $\Mod(G)$, as in Remark~\ref{remark:oY_is_frob}).
    Moreover, the $S_3$-invariance shown in Proposition~\ref{prop:Y_mult_invar} implies that this map satisfies the axiom
    \[
        \left<x,yz\right>=\left<xy,z\right>.
    \]
    Note that $\cI$ is not a Frobenius algebra because it is not self-dual: the map $\cI\ra\widetilde{\cI}$ induced from the pairing is not an isomorphism.
\end{remark}

\begin{corollary}
    The image $m(\cI\oY\cS)$ lies inside $\cS\subseteq\widetilde{\cS}$.
\end{corollary}
\begin{proof}
    Identical to Corollary~\ref{cor:mult_image_of_cI}.
\end{proof}

\begin{corollary} \label{cor:cS_is_cI_module}
    The multiplication $m\co\cS\oY\cS\ra\widetilde{\cS}$ turns $\cS$ into an $\cI$-module with respect to $\oY$.
\end{corollary}
\begin{proof}
    Identical to Corollary~\ref{cor:cI_alg_cS_module}.
\end{proof}

\begin{remark}
    We have seen that we can think of the multiplication $\cI\oY\cI\ra\cI$ as defining an algebra structure on $\cI$, and of the multiplication $\cI\oY\cS\ra\cS$ as defining an $\cI$-module structure on $\cS$. Both of these multiplications are associative in an appropriate sense. This raises the question of how to think about the associativity of the multiplication map $\cS\oY\cS\ra\widetilde{\cS}$.
    
    We answer this as follows. Note that the map $m\co\cS\oY\cS\ra\widetilde{\cS}$ defines a symmetric triality:
    \[
        \Phi^-(\cS\oY\cS\oY\cS)\ra\CC
    \]
    by the universal property of the contragradient $\widetilde{\cS}$ and Remark~\ref{remark:phi_minus_is_trace}. Our claim is that this triality respects the $\cI$-module structure of $\cS$. This follows immediately by the surjectivity of the map $e\co\one_\Ydown\ra\cI$.
\end{remark}

We dedicate the rest of this subsection to the proof of Proposition~\ref{prop:one_Y_onto_cI}.
\begin{proof}[Proof of Proposition~\ref{prop:one_Y_onto_cI}]
    Since
    \[
        \Hom\left(\one_\Ydown,S(\AA^\times/F^\times)\right)=0,
    \]
    it remains to show that the map $\one_Y\ra\cI$ is surjective. To do so, let us show that the sequence
    \[\xymatrix{
        \one_\Ydown \ar[r] & \cS \ar[r] & S(\AA^\times/F^\times) \ar[r] & 0
    }\]
    is exact.
    
    We will show this by proving that the sequence of contragradients
    \[\xymatrix{
        0 \ar[r] & \widetilde{S(\AA^\times/F^\times)} \ar[r] & \widetilde{\cS} \ar[r] & \widetilde{\one_\Ydown}
    }\]
    is exact. That is, it is sufficient to show that if $\phi\co\GL_2(F)\backslash\GL_2(\AA)\ra\CC$ is a smooth automorphic function whose corresponding Whittaker function
    \[
        W_\phi(g)=\int_{U(\AA)/U(F)}\phi(ug)\theta(u)\d{u}
    \]
    is $0$, then $\phi$ is $\SL_2(\AA)\cdot\GL_2(F)$-invariant.
    
    Indeed, from $W_\phi=0$ we see that $\phi$ is left $U(\AA)$-invariant. Since $\phi$ is also $\GL_2(F)$-invariant, it follows that $\phi$ is left-invariant with respect to the subgroup generated by $U(\AA)$ and $\GL_2(F)$. However, it is easy to check that this subgroup contains $\SL_2(\AA)$. This finishes the proof.
\end{proof}

\subsection{Abstract Automorphicity} \label{subsect:Y_alg_sheaves}

The purpose of this subsection is to introduce the category $\Mod^\aut(\GL_2(\AA))$ of abstractly automorphic representations. This category will be a full subcategory of $\Mod(\GL_2(\AA))$, containing all irreducible automorphic representations, and closed under taking subquotients and contragradients. The author's belief is that this category is a natural place in which to study the theory of automorphic representations.

Let $\Mod(\cI)$ be the category of $\cI$-modules in $\Mod(\GL_2(\AA))$. This makes sense, as $\Mod(\GL_2(\AA))$ is a symmetric monoidal category, and $\cI$ is an algebra with respect to its symmetric monoidal structure $\oY$. We begin by observing that:
\begin{claim} \label{claim:mod_cI_is_fully_faithful}
    The forgetful functor from $\Mod(\cI)$ to $\Mod(\GL_2(\AA))$ is fully faithful.
\end{claim}
\begin{proof}
    Follows immediately from Proposition~\ref{prop:one_Y_onto_cI}.
\end{proof}
This justifies giving a name to the essential image:
\begin{definition}
    Let $V\in\Mod(\GL_2(\AA))$. We will say that $V$ is \emph{abstractly automorphic} if $V$ belongs to the essential image of $\Mod(\cI)$ in $\Mod(\GL_2(\AA))$.
    
    Denote by $\Mod^\aut(\GL_2(\AA))\subseteq\Mod(\GL_2(\AA))$ the full subcategory of abstractly automorphic representations.
\end{definition}

\begin{remark} \label{remark:abst_aut_if_factors}
    Equivalently, a smooth $\GL_2(\AA)$-module $V$ is abstractly automorphic if and only if the canonical map $\one_\Ydown\oY V\ra V$ factors through the quotient $\cI\oY V$.
\end{remark}

\begin{remark}
    Note that the above also means that the structure of $\cS$ as an $\cI$-module observed in Corollary~\ref{cor:cS_is_cI_module} is unique. That is, since Corollary~\ref{cor:cS_is_cI_module} implies that $\cS\in\Mod^\aut(\GL_2(\AA))$, it follows from Claim~\ref{claim:mod_cI_is_fully_faithful} that $\cS$ carries a unique structure of $\cI$-module.
\end{remark}

\begin{proposition}
    The category $\Mod^\aut(\GL_2(\AA))$ of abstractly automorphic representations is closed under taking subquotients in $\Mod(\GL_2(\AA))$.
    
    Moreover, it is also closed under taking contragradients.
\end{proposition}
\begin{proof}
    The proof is a straightforward application of Remark~\ref{remark:abst_aut_if_factors}.
    
    Let us show closure under taking contragradients. Let $V\in\Mod^\aut(\GL_2(\AA))$. By Remark~\ref{remark:abst_aut_if_factors}, we must show that the identity map $\widetilde{V}\ra\widetilde{V}$ factors through $\cI\oY\widetilde{V}$. By the universal property of $\widetilde{V}$, this is equivalent to showing that the pairing
    \[
        \Phi^-(V\oY\widetilde{V})\ra \CC
    \]
    factors through
    \[
        \Phi^-(V\oY(\cI\oY\widetilde{V})).
    \]
    However, this holds by associativity of $\oY$ and the abstract automorphicity of $V$.
    
    Let us also show closure under taking sub-objects. Closure under quotients is similar. Let $V$ be abstractly automorphic, and let $V'\subseteq V$. Then we have a map
    \[
        \cI\oY V\ra V.
    \]
    We claim that the image of $\cI\oY V'\ra V$ lies inside $V'$. Indeed, we can check the image by pre-composing with $\one_\Ydown\oY V'\ra\cI\oY V'$, but the map $\one_\Ydown\oY V'\ra V$ factors through the canonical map $\one_\Ydown\oY V'\ra V'$.
\end{proof}

\begin{remark}
    We have already seen in Corollary~\ref{cor:cS_is_cI_module} that $\cS$ is abstractly automorphic. Therefore, so is $\widetilde{\cS}$, and in particular, all irreducible automorphic representations are also abstractly automorphic, being subquotients of it.
\end{remark}

\begin{remark} \label{remark:abst_aut_is_subspace}
    Very informally, if we imagine objects of $\Mod(\GL_2(\AA))$ as sheaves on some algebraic space, then Proposition~\ref{prop:one_Y_onto_cI} says that $\cI$ defines a closed subspace of that algebraic space.
    
    Let us elaborate on this analogy. Instead of the symmetric monoidal category $\Mod(\GL_2(\AA))$ equipped with $\oY$, consider the category $\Mod(A)$ of $A$-modules for some commutative algebra $A$, equipped with the relative tensor product over $A$. Then the unit of $\Mod(A)$ is the object $A$, and its quotient objects (as an $A$-module) are precisely the closed sub-schemes of $\Spec{A}$. In this manner, the quotient $\cI$ of the unit $\one_\Ydown$ can be thought of as defining a closed subspace.
\end{remark}

\begin{example}
    Consider the space $S(\M_2(\AA))$ of smooth and compactly supported functions on $\M_2(\AA)$, which appears as the space of test functions in Godement--Jacquet zeta integrals. This space has actions of $\GL_2(F)$ from the left and right. Consider the vector space $\cL=S(\M_2(\AA))_{/\GL_2(F)}$ of co-invariants under one of those actions. We observe that it is abstractly automorphic.
    
    Indeed, it can be written as:
    \[
        \cL=S(\M_2(\AA))_{/\GL_2(F)}=\cS\oY S(A\backslash G)
    \]
    where $A=\left\{\begin{pmatrix}\AA^\times & \\ & 1\end{pmatrix}\right\}\subseteq\GL_2(\AA)$ is a copy of $\GG_m$ inside $\GL_2$. Since $\cS$ is abstractly automorphic, then so is $\cL$.
\end{example}

\begin{appendices}
\section{Spin Structures} \label{app:spin_struct}

We have stated in the introduction to this paper (Section~\ref{sect:intro_to_triality}) that a piece of global data called an \emph{unramified spin structure} turns an appropriate space $\cI_{/\SL_2(\O_\AA)}$ of spherical automorphic functions into a commutative algebra with respect to the usual tensor product $\otimes$. In this appendix, we will formalize this statement. We will define what an unramified spin structure is, and then show how it yields a canonical multiplicative structure on $\cI_{/\SL_2(\O_\AA)}$. This will allow us to think of functions in $\cI_{/\SL_2(\O_\AA)}$ as functions on the spectrum of $\cI_{/\SL_2(\O_\AA)}$, which is the promised Fourier transformation.

As explained in Section~\ref{sect:intro_to_triality}, one may think of this result as a kind of spectral Fourier transform on the space of spherical automorphic forms
\[
    \cI_{/\SL_2(\O_\AA)}\subseteq S(\GL_2(F)\backslash\GL_2(\AA)/\SL_2(\O_\AA))
\]
which are orthogonal to one-dimensional characters. Here, $\O_\AA\subseteq\AA$ is the subring of all adeles which are integral at all places.

Let us begin with describing what an unramified spin structure is. Let $C$ be the unique smooth and proper curve with function field $F$, and let $\Omega=\Omega_C$ be its sheaf of differentials. Then $\Omega$ is a line bundle on $C$.
\begin{definition}
    An \emph{unramified spin structure} on $C$ is a line bundle $L$ on $C$, equipped with an isomorphism:
    \[
        \eta\co L\otimes L\xrightarrow{\sim}\Omega.
    \]
\end{definition}

\begin{remark}
    An unramified spin structure on $C$ always exists (see \cite{diff_has_root}).
\end{remark}

For the rest of this appendix, fix an unramified spin structure $(L,\eta)$ on $C$. Our goal is to give a construction for a commutative algebra structure on $\cI_{/\SL_2(\O_\AA)}$.

We now turn to defining the convolution product on $\cI_{/\SL_2(\O_\AA)}$. We will do this by globalizing the local construction of the co-algebra $E$ from Subsection~\ref{subsect:Y_for_sph_reps}, and taking an appropriate space of maps between $E$ and $\cI$. In fact, this will let us turn the $\SL_2(\O_\AA)$-co-invariants of an arbitrary algebra with respect to $\oY$ into an algebra with respect to $\otimes$, as in Remarks~\ref{remark:E_is_assoc} and~\ref{remark:E_comult_to_conv_prod}.

The key point is that we have made several choices in our constructions so far, and we will show how they are all fixed by the unramified spin structure $(L,\eta)$. The first piece of data is the additive character $e\co\AA/F\ra\CC^\times$ that we have chosen in Subsection~\ref{subsect:glob_oY}. Moreover, in order for the construction of Subsection~\ref{subsect:Y_for_sph_reps} to work, we needed to fix some additional local data at every place, as mentioned in Remark~\ref{remark:E_co_mult_is_spin}.

Let us derive this data from the spin structure. Choose any rational section $s\in F\otimes L$ of $L$, and suppose that the additive character $e\co\AA/F\ra\CC^\times$ we have chosen in Subsection~\ref{subsect:glob_oY} corresponds to the section $\eta(s\otimes s)\in F\otimes\Omega$ in the standard way, i.e. by identifying $F\otimes\Omega$ with the dual of $\AA/F$ by taking the sum of all residues.
\begin{remark}
    It will turn out that the multiplication on $\cI_{/\SL_2(\O_\AA)}$ is independent of the choice of rational section $s$.
\end{remark}

Observe that the valuation of $e$ at any place $v$ is necessarily even. In particular, there is a unique $\O_\AA$-sub-module $\Lambda\subseteq \AA$ which is self-dual with respect to $e$.
\begin{definition}
    Let
    \[
        {E_\AA}=\cInd_{\SL_2(\O_\AA)}^G(\one)
    \]
    be the $G$-module given by induction with compact support from the trivial representation of $\SL_2(\O_\AA)$.
\end{definition}

The self-dual lattice $\Lambda$ allows us to fix the co-multiplication on ${E_\AA}$ as follows.
\begin{construction} \label{const:E_co_mult_global}
    We define the map
    \[
        \mu\co {E_\AA}\ra {E_\AA}\oY {E_\AA}
    \]
    to be induced by the $S_3\ltimes \SL_2(\O_\AA)^3$-invariant distribution
    \[
        \one_{\M_2(\Lambda)}(g)\delta_{1}(y)
    \]
    on $\M_2(\AA)\times \AA^\times$, in a similar manner to Construction~\ref{const:E_co_mult}.
\end{construction}

\begin{claim} \label{claim:spherical_is_idemp_global}
    The map $\mu\co {E_\AA}\ra {E_\AA}\oY {E_\AA}$ of Construction~\ref{const:E_co_mult_global} turns ${E_\AA}$ into a co-unital, co-commutative and co-associative co-algebra.
\end{claim}
\begin{proof}
    Co-unitality follows immediately from the local Claim~\ref{claim:spherical_is_idemp}, using the adjustments described in Remark~\ref{remark:E_co_mult_is_spin}.
    
    Co-associativity follows from the Idempotent Theorem of \cite{cat_weil}, as in Remark~\ref{remark:E_is_assoc}.
\end{proof}

We can now state our final result for this appendix:
\begin{construction}
    The space of spherical automorphic functions:
    \[
        \cI_{/\SL_2(\O_\AA)}=\sHom({E_\AA},\cI)
    \]
    acquires a canonical quasi-unital commutative algebra structure over the center of $\Mod(G)$, with respect to the usual tensor product $\otimes$. This follows by the co-algebra structure on ${E_\AA}$ given in Claim~\ref{claim:spherical_is_idemp_global} and the algebra structure on $\cI$ given in Corollary~\ref{cor:cI_alg_cS_module}.
    
    It is possible to see (by a direct computation) that the resulting algebra structure on $\cI_{/\SL_2(\O_\AA)}$ is independent of the choice of section $s$.  
\end{construction}

\begin{remark}
    It is also possible to show that the multiplication on $\cI_{/\SL_2(\O_\AA)}$ is invariant with respect to the involution
    \begin{align*}
        \iota\co \cI & \xrightarrow{\sim} \cI \\
        \iota(f)(g) & = \abs{\det(g)}^{-2}f(g^{-T}).
    \end{align*}
    That is, if the product of $f,f'\in\cI_{/\SL_2(\O_\AA)}$ is given by $f''\in\cI_{/\SL_2(\O_\AA)}$, then the product of $\iota(f)$ and $\iota(f')$ is given by $\iota(f'')$.
\end{remark}

\end{appendices}

\Urlmuskip=0mu plus 1mu\relax
\bibliographystyle{alphaurl}
\bibliography{main}

\end{document}